\documentclass[a4paper,11pt]{amsart}

\usepackage{amscd}
\usepackage{xypic}  
\usepackage{amssymb}
\usepackage{amsthm}
\usepackage{epsfig}
\usepackage[T1]{fontenc}
\usepackage{color}
\usepackage{verbatim}

\newtheorem{thm}{Theorem}[section]
\newtheorem*{thm*}{Theorem}
\newtheorem{cor}[thm]{Corollary}
\newtheorem{lemma}[thm]{Lemma}
\newtheorem{prop}[thm]{Proposition}
\newtheorem{defn}[thm]{Definition}

\newtheorem{proposition}[thm]{Proposition}
\newtheorem{conjecture}[thm]{Conjecture}
\newtheorem*{conjecture*}{Conjecture}

\theoremstyle{definition}
\newtheorem{remark}[thm]{Remark}

 \newtheorem{example}[thm]{Example}

\def\ker{\operatorname{ker}}

\def\codim{\operatorname{codim}}
\def\min{\operatorname{min}}
\def\im{\operatorname{Im}}

\def\max{\operatorname{max}}

\def\c1{\operatorname{c_1}}
\def\c2{\operatorname{c_2}}

\def\Hilb{\operatorname{Hilb}}
\def\Alb{\operatorname{Alb}}
\def\Sym{\operatorname{Sym}}
\def\Mon{\operatorname{Mon}}

\def\rk{\operatorname{rk}}

\def\div{\operatorname{div}}
\def\GCD{\operatorname{GCD}}

\def\gg{\mathfrak{g}}
\def\CC{{\mathbb C}}
\def\ZZ{{\mathbb Z}}

\def\DD{{\mathbb D}}
\def\QQ{{\mathbb Q}}
\def\PP{{\mathbb P}}
\def\LL{{\mathbb L}}
\def\A{{\mathcal A}}

\def\G{{\mathcal G}}

\def\L{{\mathcal L}}
\def\M{{\mathcal M}}
\def\N{{\mathcal N}}
\def\O{{\mathcal O}}
\def\I{{\mathcal J}}

\def\E{{\mathcal E}}
\def\T{{\mathcal T}}

\def\F{{\mathcal F}}
\def\K{{\mathcal K}}

\def\C{{\mathcal C}} 
\def\K{{\mathcal K}}

\def\P{{\mathcal P}}

\def\x{\times}                   
\def\eps{\varepsilon}  
\def\cong{\simeq}

\def\sub{\subseteq}

\def\+{\oplus}                   
\def\*{\otimes}                  

\def\Hom{\operatorname{Hom}}
\def\hom{\operatorname{ \mathfrak{hom}}}
\def\ext{\operatorname{ \mathfrak{ext}}}

\def\Pic{\operatorname{Pic}}
\def\NS{\operatorname{NS}}
\def\Supp{\operatorname{Supp}}

\def\Supp{\operatorname{Supp}}
\def\det{\operatorname{det}}
\def\Sing{\operatorname{Sing}}

\def\Def{\operatorname{Def}}
\def\tC{\widetilde{C}}
\def\tS{\widetilde{S}}
\newcommand{\kahl}{K\"{a}hler }
\begin{document}

\title[Wall divisors and algebraically coisotropic subvarieties of IHS manifolds]{Wall divisors and algebraically coisotropic subvarieties of irreducible holomorphic symplectic manifolds} 

\author{Andreas Leopold Knutsen}
\address{Andreas Leopold Knutsen, Department of Mathematics, University of Bergen,
Postboks 7800,
5020 BERGEN, Norway}
\email{andreas.knutsen@math.uib.no}

\author{Margherita Lelli-Chiesa}
\address{Margherita Lelli-Chiesa, Centro di Ricerca Matematica Ennio De Giorgi, Scuola Normale Superiore, Piazza dei Cavalieri 3, 56100 Pisa, Italy}
\email{margherita.lellichiesa@sns.it}

\author{Giovanni Mongardi}
\address{Giovanni Mongardi, Department of Mathematics, University of Milan, via Cesare Saldini 50, 20133 Milan, Italy}
\email{giovanni.mongardi@unimi.it}

\begin{abstract} 
Rational curves on Hilbert schemes of points on $K3$ surfaces and generalised Kummer manifolds are constructed by using Brill-Noether theory on nodal curves on the underlying surface. It turns out that all wall divisors can be obtained, up to isometry, as dual divisors to such rational curves. The locus covered by the rational curves is then described, thus exhibiting algebraically coisotropic subvarieties. This provides strong evidence for a conjecture by Voisin concerning the Chow ring of irreducible holomorphic symplectic manifolds. Some general results concerning the birational geometry of irreducible holomorphic symplectic manifolds are also proved, such as a non-projective contractibility criterion for wall divisors. 
\end{abstract}

\maketitle


\section{Introduction}

Rational curves play a pivotal role in the study of the birational geometry and the Chow ring of algebraic varieties. The present paper concerns a specific class of varieties, namely, {\it irreducible holomorphic symplectic (IHS) manifolds} and, more precisely, {\it Hilbert schemes of points on $K3$ surfaces} and {\it generalised Kummer manifolds} (cf. \S\ref{sec:due}), and is focused on some special rational curves arising from the Brill-Noether theory of normalisations of curves lying on $K3$ and abelian surfaces. In order to treat the two cases simultaneously, we introduce the following notation: we set $\eps=0$ (respectively, $\eps=1$) when $S$ is a $K3$ (resp., abelian) surface, and we denote by $S^{[k]}_\eps$ the Hilbert scheme of $k$ points on $S$ when $\eps=0$ and the $2k$-dimensional generalised Kummer variety on $S$ when $\eps=1$.



In the last few years, some classical results concerning ($-2$)-curves on $K3$ surfaces have been generalised to higher dimension and in particular it was shown that rational curves fully control the birational geometry of IHS manifolds.  More precisely, Ran \cite{ran} proved that extremal rational curves can be deformed together with the ambient IHS manifold, and this was exploited by Bayer, Hassett and Tschinkel \cite{BHT} in order to determine the structure of the ample cone. The same result was independently obtained by the third named author \cite{Mo} using intrinsic properties of IHS manifolds and a deformation invariant class of divisors, the so-called {\it wall divisors} (cf. Definition \ref{def:wall}), which contains all divisors dual to extremal rays. This class of divisors was also studied by Amerik and Verbitsky \cite{AV}, who investigated fibres of extremal contractions. Indeed, the MBM classes in \cite{AV} turn out to be precisely the dual curve classes to wall divisors, cf. Remark \ref{rem:MBM}.

By results of Bayer and Macr\`i \cite{bm, bm2} and Yoshioka \cite{yoshi}, moduli spaces of stable objects in the bounded derived category of a $K3$ or abelian surface $S$ provide examples of deformations of $S^{[k]}_\eps$ and the space of stability conditions can be used towards computing their ample cones.  

In this paper we use Brill-Noether theory of nodal curves on abelian and $K3$ surfaces in order to exhibit rational curves in $S^{[k]}_\eps$ and describe, in many cases, the locus they cover.  Our construction proceeds as follows. Let $(S,L)$ be a general primitively polarized $K3$ or abelian surface of genus $p:=p_a(L)$ and let $C \in |L|$ be a $\delta$-nodal curve whose normalization $\widetilde{C}$ has a linear series of type $g^1_{k+\eps}$. Existence of a family of such curves having the expected dimension (and satisfying certain additional properties) has been proved in 
\cite{ck,klm} under suitable conditions on the triple $(p,k,\delta)$, cf. Theorem \ref{thm:exist}. Any pencil of degree $k+\eps$ on $\widetilde{C}$ defines a rational curve in $S^{[k]}_\eps$, whose class is 
\[
R_{p,\delta,k}:=L-(p-\delta+k-1+\eps)\mathfrak{r}_k,
\]
in terms of the canonical decomposition $N_1 (S_{\eps}^{[k]}) \cong N_1(S) \+ \ZZ[\mathfrak{r}_k]$, cf. \eqref{eq:N1} and Lemma \ref{juveladra}.
In particular, its Beauville-Bogomolov square is easily computed to be
\[
q(R_{p,\delta,k})= 2(p-1)-\frac{(p-\delta+k-1+\eps)^2}{2(k-1+2\eps)},
\]
cf. \eqref{eq:squarecurve}. An important additional feature of the rational curves obtained in this way is that they move in a family of dimension precisely $2k-2$ in $S^{[k]}_\eps$ and thus survive in all small deformations of $S^{[k]}_\eps$ that keep $R_{p,\delta,k}$ algebraic. 

We prove the following result concerning the dual (in the sense of the lattice duality induced by the Beauville-Bogomolov form) divisor $D_{p,\delta,k}$ to the class $R_{p,\delta,k}$.

\begin{thm} (cf. Theorem \ref{thm:muri_curve_doppio})\label{zero}
The divisor $D_{p,\delta,k}$ is a wall divisor if and only if $q(R_{p,\delta,k})<0$.
\end{thm}

 By comparison with \cite{bm2,yoshi}, we show that all wall divisors are realized as $D_{p,\delta,k}$ for some integers $p$, $\delta$ and $k$, up to  isometry (in the sense of lattice theory), cf. Proposition \ref{prop:tutti_muri}. This is rather striking, as it shows that the birational geometry of $S^{[k]}_\eps$ can be recovered from classical Brill-Noether theory of curves on the underlying surface, at least when the monodromy group is maximal.  We mention that some wall divisors have also been recently constructed by Hassett and Tschinkel \cite{HT_new}, using a different approach.  

Under opportune assumptions, we explicitly construct the locus in $S^{[k]}_\eps$ covered by our rational curves of class $R_{p,\delta,k}$. When $D_{p,\delta,k}$ is a wall divisor this locus may be described abstractly using only lattice theoretic properties, as in \cite{bm,yoshi} and in the more recent \cite{HT_new}. However, our constructions  only rely on the definition of our curves of class $R_{p,\delta,k}$ and are thus very concrete. 

The first type of construction goes as follows. Let $\mathcal{M}$ be the component of the moduli space of (Gieseker) $L$-stable torsion free sheaves on $S$ with Mukai vector $v=(2, c_1(L),\chi+2(\eps-1))$
containing the Lazarsfeld-Mukai bundle associated with the pushforward to a $\delta$-nodal curve in $S$ of a $g^1_{k+\eps}$ on its normalization. As soon as $\chi:=p-\delta-k+3-5\eps \geq 2\delta+2$,
we construct a variety $\mathcal{P}\to \mathcal{M}\times S^{[\delta]}$ which is generically a projective bundle. The fibre of $\mathcal{P}$ over a point $([\E], \tau)\in \mathcal{M}\times S^{[\delta]}$ is the projectivization  of the space of global sections of $\E$ vanishing along $\tau$. We then define a rational map $g:\mathcal{P}\dashrightarrow S^{[k]}_\eps$ and denote by $T$ the closure of the image of $g$, which is an irreducible component of the locus covered by curves of class $R_{p,\delta,k}$. We show that $g$ is birational, thus obtaining the following:
\begin{thm} (cf. Theorem \ref{thm:contrazioni})\label{uno}
Let $(S,L)$ be a very general primitively polarized $K3$ or abelian surface of genus $p \geq 2$. Let $k \geq 2$ and $0 \leq \delta \leq p-\eps$ be integers such that 
$$\max\{2\delta+2,4\eps\}\leq\chi:=p-\delta-k+3-5\eps \leq \delta+k+1.$$
Then, there is a subscheme $T\subset  S_{\eps}^{[k]}$  birational to a $\mathbb{P}^{\chi-2\delta-1}$-bundle on a holomorphic symplectic manifold $W$ of dimension 
$2(k+1+2\delta-\chi)$. Furthermore, the lines contained in any fibre of the rational projection $T\dashrightarrow W$ have class $R_{p,\delta,k}$.
\end{thm}

 The resulting uniruled subvarieties are contractible (up to birational equivalence) when the curve $R_{p,\delta,k}$ has negative square. In the case where $\delta=0$ and $R_{p,\delta,k}$ has the minimal possible Beauville-Bogomolov square, namely, $-(k+3-2\eps)/2$, we use Theorem \ref{uno} in order to construct a Lagrangian $k$-plane $\mathbb{P}^k\subset  S_{\eps}^{[k]}$ such that $R_{p,\delta,k}$ is the class of its lines, cf. Example \ref{ex:romarosica} and Proposition \ref{theend}. This agrees with Bakker's result \cite[Thm. 3]{Ba} stating that, in the case $\eps=0$, a primitive class generating an extremal ray is the line in a Lagrangian $k$-plane if and only if its square is  $-(k+3)/2$, and suggests that the analogous statement should hold for $\eps=1$. Note that very few examples of Lagrangian planes are explicitly described in the literature, cf. \cite[Ex. 8, 9, 10]{Ba}.  
 
Our rational curves have applications to the Chow ring of IHS manifolds, too. In the recent paper \cite{voi}, Voisin stated the following:
\begin{conjecture}(cf. \cite[Conj. 0.4]{voi})\label{voisin}
Let $X$ be a projective IHS manifold of dimension $2k$ and let $\mathbb{S}_r(X)$ be the set of points in $X$ whose orbit under rational equivalence has dimension at least $r$. Then $\mathbb{S}_r(X)$ has dimension $2k-r$.  
\end{conjecture}

The above sets $\mathbb{S}_r(X)$ are countable unions of closed algebraic subsets of $X$ and endow the Chow group $\mathrm{CH}_0(X)$ of $0$-cycles with a filtration $\mathbb{S}_\bullet$ which is conjecturally connected with the Bloch-Beilinson filtration and its splitting predicted by Beauville \cite{B1}. The question about non-emptiness of $\mathbb{S}_r(X)$ is still open and related to the existence problem for algebraically coisotropic subvarieties of $X$. If $X$ has dimension $2k$ and $\sigma$ is its symplectic form, a subvariety $Y\subset X$ of codimension $r$ is {\it algebraically coisotropic} if there exist a ($2k-2r $)-dimensional variety $B$ and a surjective rational map $Y\dashrightarrow B$ such that $\sigma|_Y$ is the pullback of a $2$-form on $B$. The subvarieties $T\subset S_{\eps}^{[k]}$ of Theorem \ref{uno} are algebraically coisotropic by construction and they are components of $\mathbb{S}_r(S_{\eps}^{[k]})$ of dimension $2k-r$, with $r:=\chi-2\delta-1$ (cf. Corollary \ref{cor1}). Starting from $T$ and then applying the natural rational map $S^{[k+\eps]} \x S^{[l-k]} \dasharrow S^{[l+\eps]}$, one obtains a component of $\mathbb{S}_r(S_{\eps}^{[l]})$ for any $l\geq k$. We use this observation in Theorem \ref{thm:contrazioni2} in order to construct components of $\mathbb{S}_r(S_{\eps}^{[k]})$, with $k$ fixed, for several values of $r$.

 Our second  construction of uniruled subvarieties of $S^{[k]}_{\eps}$ is obtained by considering the Severi variety $V_{\{L\},\delta}$ of $\delta$-nodal curves in the continuous system $\{L\}$ with $\delta$ big enough; the assumptions on $\delta$ ensure, in particular, that the normalization $\widetilde{C}$ of any curve in $V_{\{L\},\delta}$ has a $g^1_{k+\eps}$. For any integer $k'$ satisfying suitable conditions, the symmetric product $\Sym^{k'+\eps}(\widetilde{C})$ is generically a $\PP^r$-bundle on $\Pic^{k'+\eps}(\widetilde{C})$, where $r$ depends on the integers $\delta$ and $k'$. By varying them, we exhibit ($2k-r$)-dimensional components of $\mathbb{S}_r(S_{\eps}^{[k]})$ for any $r$, except $r=k$ when $\eps=1$. More precisely, we prove the following:
\begin{thm}(cf. Theorem \ref{thm:nuovofigo})\label{due}
  Let $(S,L)$ be a general primitively polarized $K3$ or abelian surface of genus $p \geq 2$ and fix an integer $k \geq 2$. Then for any integer $r$ such that $1 \leq r \leq k-\eps$, and any integer $k'$ such that $r+\eps \leq k' \leq \min\{k,p+r-\eps\}$,  the set $\mathbb{S}_r(S_{\eps}^{[k]})$ has an irreducible component $W_{r,k'}$ satisfying the following:
  \begin{itemize}
  \item[(i)] $\dim W_{r,k'}=2k-r$;
  \item[(ii)] $W_{r,k'}$ is birational to a $\mathbb{P}^{r}$-bundle and hence algebraically coisotropic;
  \item[(iii)] the class of the lines in the $\PP^r$-fibres is $L-[2(k'+\eps)-r-1]\mathfrak{r}_k$;
  \item[(iv)] the maximal rational quotient of the desingularization of $W_{r,k'}$ has dimension  $2(k-r)$.
 \end{itemize}
\end{thm}
Point (iv) positively answers, in the case of $S_{\eps}^{[k]}$, a question by Charles and Pacienza (cf. \cite[Question 1.2]{cp}) concerning existence of subvarieties of an IHS manifold whose maximal rational quotients have the minimal possible dimension. 

For $\eps=0$, examples of ($2k-r$)-dimensional components of $\mathbb{S}_r(S_{\eps}^{[k]})$ for any $r$ were already provided in \cite[\S 4.1 Ex. 1 and Lemma 4.3]{voi} by considering fibres of the Hilbert-Chow morphism $\mu_k: S^{[k]} \rightarrow \Sym^{k}(S)$. However, our components $W_{r,k'}$ are not contained in the exceptional locus of $\mu$ and thus provide much stronger evidence for Conjecture \ref{voisin}. 

In developing techniques towards proving the above theorems, we obtain some general results on IHS manifolds. First of all, in Proposition \ref{prop:qualisono} we provide a criterion to tell whether a deformation of $S^{[k]}_\eps$ is isomorphic to $S'^{[k]}_\eps$ for some surface $S'$. This appears to be related to ideas from \cite{add} and \cite{mw}. Secondly, we prove that wall divisors can be contracted under general assumptions:

\begin{thm} (cf. Theorem \ref{thm:mazzone})\label{mazzo}
Let $X$ be a projective IHS manifold and let $D$ be a wall divisor on $X$. Then one of the following holds:
\begin{itemize}
\item There exists a curve $R$ dual to $D$ that moves in a divisor and a birational map $X\dashrightarrow Y$ contracting $R$. Moreover $Y$ is singular symplectic.
\item For a general deformation $(X',D')$ of $(X,D)$, there is a birational map $X'\dashrightarrow X''$ with $X''$ IHS and a contraction $X''\rightarrow Y$ that contracts all curves dual to $D'$.
\end{itemize}
\end{thm}

This result holds in particular for general nonprojective deformations of $(X,D)$, where a proof of the contraction theorem was, as yet, unavailable.

The paper is organized as follows. Section \ref{sec:due} contains background material concerning IHS manifolds and in particular varieties of the form $S^{[k]}_\eps$. In Section \ref{sec:tre} we recall known results on the birational geometry of IHS manifolds and use them in order to prove Theorem \ref{mazzo}. We then specialize to deformations of $S^{[k]}_\eps$ and prove that $-(k+3-2\eps)/2$ is a lower bound for the self-intersection of a primitive generator of an extremal ray of the Mori cone, cf. Proposition \ref{prop:bound}; the result is new for $\eps=1$, while it had already appeared in \cite{BHT,Mo} for $\eps=0$. 

Section \ref{sec:quattro} summarises the results from \cite{ck,klm} concerning the Brill-Noether theory of nodal curves on symplectic surfaces. Classes $R_{p,\delta,k}$ are computed. Proposition \ref{prop:andreas010715} proves the existence of a family of rational curves of class $R_{p,\delta,k}$ having the expected dimension and surviving in any small deformation of $S^{[k]}_\eps$ that keeps the class algebraic. 
In Section \ref{sec:cinque} we prove Theorem \ref{zero} and exhibit a collection of wall divisors that we later show to be essentially ``complete'' in Proposition \ref{prop:tutti_muri}. 

 Section \ref{sec:LM} proves several results concerning vector bundle techniques associated with nodal curves, which are essential in the proof of Theorem \ref{uno}. We  believe that these results are of independent interest, due to the recent activity in the study of nodal curves on $K3$ and abelian surfaces. In particular, Proposition \ref{prop:pareschi} extends a result by Pareschi \cite[Lemma 2]{P} to possibly nodal curves on symplectic surfaces; Proposition \ref{prop:importante} and Lemma \ref{h1} describe properties of general (semistable) sheaves lying in a specific component $\M$ of their moduli space.

The main results Theorems \ref{uno}, and \ref{due} are finally proved in Section \ref{sec:sei}.

\vspace{0.3cm}
\textbf{Note.} After this paper was completed, a paper by H.~Y.~Lin \cite{lin} appeared on the arXiv, where the author also constructs components of the locus $\mathbb{S}_r$ for generalised Kummer manifolds. Our constructions are different from the Lin's and the spirit of the two papers is quite distant.

\section*{Acknowledgements}
We are grateful to C.~Ciliberto and K.~O'Grady for interesting conversations on this topic. Moreover, we thank the Max Planck Institute for mathematics and the Hausdorff Center for Mathematics  in Bonn, the University of Bonn and the Universities of Roma La Sapienza, Roma Tor Vergata and Roma Tre, for hosting one or more of the authors at different times enabling this collaboration. The second named author was supported by the Centro di Ricerca Matematica Ennio De Giorgi in Pisa and the third named author by ``Firb 2012, Spazi di moduli ed applicazioni''.

\section{Generalities on IHS manifolds}\label{sec:due}

A compact K\"{a}hler manifold $X$ is called {\it hyperk{\"a}hler} or 
{\it irreducible holomorphic symplectic (IHS)} if it is simply connected and $H^0(\Omega_X^2)$ is generated by a symplectic form.

The symplectic form implies the existence of a canonical quadratic form $q(\,)$ on $H^2(X,\mathbb{Z})$, called the {\it Beauville-Bogomolov form}, and of a constant $c$, the {\it Fujiki constant}, such that for every $\alpha\in H^2(X,\mathbb{Z})$ one has:
\begin{equation}
q(\alpha)^n=c\cdot \alpha^{2n},
\end{equation}
where $\dim(X)=2n$. We will denote by $b(\,,\,)$ the bilinear form associated with $q$. This endows $H^2(X,\mathbb{Z})$ with the structure of a lattice of signature $(3,b_2(X)-3)$ and provides an embedding of $H_2(X,\mathbb{Z})$ in $H^2(X,\mathbb{Q})$ as the usual lattice embedding $\LL^\vee\hookrightarrow \LL\otimes\mathbb{Q}$. Fora any $D\in H^2(X,\mathbb{Z})$ denote by $\div(D)$ the positive generator of the ideal $b(D,H^2(X,\mathbb{Z}))$; then the elements $D/\div(D)$, with $D$ running among all primitive elements in $H^2(X,\mathbb{Z})$, generate $H_2(X,\mathbb{Z})$.
The quadratic form and the symplectic form also allow to define a period domain for IHS  manifolds, much as in the case of $K3$ surfaces, as follows. For any lattice $\LL$, one defines the period domain
\begin{equation}\nonumber
\Omega_{\LL}:=\{\omega\,\in\,\mathbb{P}(\LL\otimes\mathbb{C})\,|\,q(\omega)=0,\,b(\omega,\overline{\omega})>0\}.
\end{equation}
Any isometry $f: H^2(X,\mathbb{Z})\,\rightarrow \LL$ is called a marking and there is a natural map, the period map 
$\P$, sending a marked IHS  manifold $(X,f)$ to $\P(X,f):=[f(\sigma_X)]\in \Omega_{\LL}$, where $\sigma_X$ is any symplectic form on $X$. Let $\M_{\LL}$ be the moduli space of deformation equivalent marked IHS  manifolds with $H^2$ isometric to $\LL$. The period map $\P\,:\,\M_{\LL}\,\rightarrow \,\Omega_{\LL}$ is surjective \cite[Thm. 8.1]{H1} and it is a local isomorphism \cite[Thm. 5]{B}.

There are singular analogues of IHS  manifolds, called symplectic varieties. A normal variety $Y$ is {\it symplectic} if it has a unique (up to scalars) nondegenerate symplectic form on its smooth locus and a resolution of singularities $\pi: \widetilde{X} \to X$ such that the pullback of this form is everywhere defined, but possibly degenerate. If it is nondegenerate, then $\widetilde{X}$ is IHS and we say that $\pi$ is a {\it symplectic resolution}. Symplectic varieties share many properties with IHS manifolds, especially when they admit a symplectic resolution. In this case it is indeed possible to define a quadratic form on their second cohomology group and the following results hold.

\begin{thm} \label{thm:nami} 
{\rm (Namikawa \cite[Thm. 2.2]{nami})}
Let $\pi\,:\,\widetilde{X}\,\rightarrow\,X$ be a symplectic resolution of a projective symplectic variety $X$. Then the Kuranishi spaces $\Def(X)$ and $\Def(\widetilde{X})$ are both smooth and of the same dimension. There exists a natural map $\pi_{*}\,:\,\Def(\widetilde{X})\,\rightarrow\, \Def(X)$ which is a finite covering. Moreover, $X$ has a flat deformation to an IHS  manifold. Any smoothing of $X$ is an IHS  manifold obtained as a flat deformation of $\widetilde{X}$.
\end{thm}

\begin{thm}\label{thm:kirch}
{\rm (Kirchner)} Let $X$ be a normal symplectic variety admitting a symplectic resolution of singularities and such that $\codim(\Sing X)\geq 4$. Let $\Def(X)_{lt}$ denote the Kuranishi space of locally trivial deformations of $X$. Then there is a well defined period map $\P: \Def(X)_{lt}\,\rightarrow\,\Omega_{\LL}$, where $\LL\simeq H^2(X,\mathbb{Z})$, having injective tangent map. 
\end{thm}

\begin{proof}
Locally trivial deformations are parametrized by a locally closed subset of $\Def(X)$. The latter is smooth by Theorem \ref{thm:nami}. After replacing $X$ with a small  locally trivial deformation, we can suppose that $\Def(X)_{lt}$ is smooth, therefore 
\cite[Cor. 3.4.2]{kirch} applies and first order locally trivial deformations are parametrised by $H^1(X-\Sing X,\Omega^1_X)\cong H^{1,1}(X)$. Now \cite[Thm. 3.4.4]{kirch} provides the period map as stated above.
\end{proof}

\begin{remark}\label{oss:loc_triv}
Keep notation as above and let $R_1,\dots,R_i$ be the curve classes that span the classes of curves contracted by the resolution of singularities $\widetilde{X}\rightarrow X$. The above theorem implies that first order locally trivial deformations of $X$ are parametrised by $H^{1,1}(\widetilde{X})\cap \langle R_1,\dots,R_i \rangle^\perp\cong H^{1,1}(X)$.
\end{remark}

Very few examples of IHS  manifolds are known. The present paper will focus on the two infinite families of examples introduced by Beauville \cite{B}, namely, {\it Hilbert schemes of points on $K3$ surfaces} and {\it generalised Kummer manifolds}.
Let $S$ be a $K3$ or abelian surface. Throughout the paper we will let
\begin{equation} \label{eq:defeps}
 \eps=\eps_S:=
\begin{cases}  
1 & \;\mbox{if $S$ is abelian}, \\
0 & \;\mbox{if $S$ is $K3$}.
\end{cases} 
\end{equation}
It was proved by Beauville \cite{B} that the Hilbert scheme  $S^{[k+\eps]}$
of $0$-dimensional  subschemes of $S$ of length $k+\eps$ , where $k \geq 2$, inherits a symplectic form from $S$ and is smooth.  When $S$ is $K3$, it is simply connected and thus an IHS manifold of dimension $2k$. When $S$ is abelian, $S^{[k+1]}$ is not simply connected, but  any fibre of the Albanese map $\Sigma_k: S^{[k+1]}\rightarrow \Alb S^{[k+1]} \cong S$ is a $2k$-dimensional  IHS  manifold $K^k(S)$, which is called a generalised Kummer manifold. We recall that $\Sigma_k$ is the composition of the Hilbert-Chow morphism $\mu_k: S^{[k+\eps]} \rightarrow \Sym^{k+\eps}(S)$ and the summation map $+: \Sym^{k+\eps}(S) \rightarrow S$. 

In order to handle the two families simultaneously, we set
\begin{equation}
\label{eq:defs}
 S_{\eps}^{[k]}:=
\begin{cases}  
 K^k(S)& \;\mbox{if $\eps=1$ (i.e., $S$ is abelian)}, \\
 S^{[k]}& \;\mbox{if $\eps=0$ (i.e., $S$ is $K3$)}.
\end{cases} 
\end{equation}
Note that $\dim S_{\eps}^{[k]}=2k$ in both cases, even though $S_{1}^{[k]} \subset S^{[k+1]}$. By abuse of notation, in the latter case we will still use the same symbol $\mu_k$ and the same name for the restriction of the Hilbert-Chow morphism to $S_{1}^{[k]}$.

There are natural embeddings
\begin{align}\label{k_immersion}
\NS (S)& \hookrightarrow \Pic (S_{\eps}^{[k]}),\\
\label{imm2}N_1(S)& \hookrightarrow N_1(S_{\eps}^{[k]}).
\end{align}
The former is given by associating with the class of a prime divisor $D$ in $S$ the divisor
\begin{equation}\label{imm3}
\{Z \subset S_{\eps}^{[k]}\,|\, \Supp(Z) \cap D \neq \emptyset\}
\end{equation}
and the latter is given by fixing  a set of general points $\{x_1,\ldots, x_{k+\eps-1}\}\subset S$ and associating with the class of an effective curve $C\subset S$ the class of the curve
\[
\{Z\subset S_{\eps}^{[k]} \,|\, \Supp(Z)\cap C \neq \emptyset,\,\,\{x_1,\ldots, x_{k+\eps-1}\}\subset \Supp(Z)\}.
\]
The exceptional divisor $\Delta_k$ of the Hilbert-Chow morphism $\mu_k$ has class $2\mathfrak{e}_k$ and
one has an orthogonal decomposition with respect to $b(\,,\,)$:
\begin{equation*}
H^2(S_{\eps}^{[k]},\ZZ) \cong H^2(S,\ZZ) \+_{\perp} \ZZ[\mathfrak{e}_k],
\end{equation*}
such that $b(\,,\,)$ restricts to the usual cup product on $S$ and  $q(\mathfrak{e}_k)=-2(k-1+2\eps)$. The above isometry restricts to the embedding \eqref{k_immersion} on the algebraic part, whence 
\begin{equation}
  \label{eq:pic}
 \Pic (S_{\eps}^{[k]}) \cong \NS(S) \+ \ZZ[\mathfrak{e}_k].
\end{equation}
Under the embedding $H_2(S_{\eps}^{[k]},\ZZ) \hookrightarrow H^2(S_{\eps}^{[k]},\QQ)$ given by lattice duality, $H_2(S_{\eps}^{[k]},\ZZ)$ is generated by $H^2(S,\ZZ)$ and $\mathfrak{r}_k:=\mathfrak{e}_k/2(k-1+2\eps)$. Here $\mathfrak{r}_k$ is the class of a  general rational curve lying in the exceptional divisor $\Delta_k$ of the Hilbert-Chow morphism,  that is, $\mathfrak{r}_k$ is the inverse image under $\mu_k$ of a cycle in $\Sym^{k+\eps}(S)$ supported at precisely $k-1+\eps$ points. Hence, $\div(\mathfrak{e}_k)=2(k-1+2\eps)$ and
\begin{equation}
  \label{eq:N1}
 N_1 (S_{\eps}^{[k]}) \cong N_1(S) \+ \ZZ[\mathfrak{r}_k]. 
\end{equation}

Any smooth K{\"a}hler deformation of $S_{\eps}^{[k]}$ is called a  manifold {\it of Kummer type} if $\eps=1$ and {\it of $K3^{[k]}$ type} if $\eps=0$.

\begin{remark} \label{rem:hodge}
The manifold $S_{\eps}^{[k]}$ can also  be defined by means of moduli spaces of stable sheaves on the underlying surface. There is a natural map $\mathcal{C}oh(S)\,\rightarrow\,H^{2*}(S,\mathbb{Z})$ sending a sheaf $\F$ to its {\it Mukai vector} 
\begin{equation}
  \label{eq:mukaivector}
v(\F):=  \mathrm{ch}(\F)\sqrt{\mathrm{td}(S)}= (\rk \F, c_1(\F), \chi(\F)+(\eps-1)\rk \F).
\end{equation}
 In order to construct a moduli space of sheaves, one needs also a choice of a polarization $L$ and, for most choices of $v$ (see \cite[Thm. 0.1]{yo_unpub}), a general ample $L$ gives a smooth moduli space $\M(v)$ of Gieseker $L$-semistable torsion free sheaves with Mukai vector $v$. Moreover, the fibre of $\M(v)$ under the Albanese map is deformation equivalent to $S_{\eps}^{[k]}$.
 
If $v:=(1,0,1-2\eps-k)$, every element $[\F]\in\M(v)$ can be written as $\F=H_0\otimes I_Z$ with $H_0\in\Pic^0(S)$ and $[Z] \in S^{[k+\eps]}$. Hence, one has $\M(v) \cong S^{[k]}$ in the $K3$ case, while in the abelian case $K^k(S)$ is the fibre over $0$ of the Albanese map of $\M(v)$,  cf. \cite[Thm. 0.1]{yo1}. 

For $k \geq 2$, we have a canonical Hodge isometry 
\[ 
H^{2}(S_{\eps}^{[k]},\mathbb{Z}) \cong H^2(S,\ZZ) \+_{\perp} \ZZ[\mathfrak{e}_k] \cong v^{\perp} \subset H^{2*}(S,\mathbb{Z})=\Lambda:=U^{\oplus 4}\oplus E_8(-1)^{\oplus2-2\eps},
\]
 such that $\mathfrak{e}_k$ is sent to $(1,0, k-1+2\eps)$ and the second cohomology of $S$ is sent back to itself,
cf. \cite[Thm. 0.2]{yo1}. 
In particular, one has
\begin{equation}\label{ercapitano}
\frac{v+\mathfrak{e}_k}{2}\in \Lambda \; \; \mbox{and} \; \; \frac{v-\mathfrak{e}_k}{2(k-1+2\eps)}\in \Lambda.
\end{equation}
\end{remark}

\section{Birational geometry and wall divisors of IHS  manifolds}\label{sec:tre}
Having trivial canonical bundle, IHS  manifolds are minimal in the sense of MMP. Therefore, maps between IHS  manifolds are rather rigid, as the following shows:

\begin{prop}
Let $X$ and $X'$ be two IHS  manifolds and let $f\,:\,X\dashrightarrow\,X'$ be a birational map. Then the following hold:
\begin{itemize}
\item[(i)] The manifolds $X$ and $X'$ are deformation equivalent and $H^2(X,\mathbb{Z})\cong H^2(X',\mathbb{Z})$ as Hodge structures.
\item[(ii)] The map $f$ has indeterminacy locus of codimension at least $2$.
\item[(iii)] If $X$ is projective, there exists a klt divisor $D$ such that the map $f$ is a sequence of flips obtained by running the minimal model program for the pair $(X,D)$.
\end{itemize}
\end{prop}

\begin{proof}
Item (i) is the content of \cite[Thm. 4.6]{H1}, and (ii) is proved in \cite[Rem. 4.4]{H1} and holds true for all manifolds with nef canonical divisor. For (iii), any (sufficiently small) multiple of an effective divisor on a IHS  manifold is klt (see \cite[Rem. 12]{HT2}). Therefore, if we take an ample divisor $A$ on $X'$ and set $D=\epsilon f^{*}(A)$, for $\epsilon<<1$, we have a klt pair $(X,D)$. As $A$ is ample and $f$ is well defined on divisors, $D$ is positive on all curves $C$ such that $\mathrm{Locus}(\mathbb{R}^{+}[C])$\footnote{We recall that the {\it locus} of  $V\subset N_1(X)$  is the closure of the locus  in $X$ covered by curves of class lying in $ V$, that is, $\mathrm{Locus}(V):=\overline{\{x\in \Gamma\subset X\,:\,[\Gamma]\in V\}}$.} is a divisor. Therefore, by running the MMP for $(X,D)$ we do not encounter any divisorial contraction. As $f_{*}D$ is ample, $(X',A)$ is a minimal model for $(X,D)$. 
\end{proof}

We refer to \cite[Thm. 4.1]{LP} for the termination of the log-MMP for IHS manifolds.

Being well-defined on divisors, any  birational map between two IHS manifolds induces a pullback map between their second cohomology groups. This allows  to define a birational invariant called the {\it birational K\"{a}hler cone} of an IHS manifold $X$. We recall that the {\it positive cone} $\C_X$ is the connected component containing a K\"{a}hler class of the cone of positive classes inside $H^{1,1}(X,\mathbb{R})$. It contains the {\it K\"{a}hler cone} $\K_X$, which is the cone 
containing all K\"{a}hler classes. The birational K\"{a}hler cone $\mathcal{BK}_X$ is the union $\cup f^{-1} \mathcal{K}_{X'}$, where $f$ runs through all birational maps between $X$ and any IHS  manifold $X'$. If $X$ is projective, then the closure of the algebraic part of the birational K\"{a}hler cone is just the movable cone, that is, the closure of the cone of divisors whose linear systems have no divisorial base components.

We recall that an isomorphism $H^2(X,\mathbb{Z})\stackrel{\simeq}{\longrightarrow} H^2(Y,\mathbb{Z})$, where $X$ and $Y$ are two IHS  manifolds, is called a parallel transport operator if it is induced by the parallel transport in the local system $R^2\pi_*\ZZ$  along a path of smooth deformations $\pi\,:\,\mathcal{X}\rightarrow \DD$ over a disc $\DD$ such that $X$ and $Y$ are two fibres.  The group of parallel self-operators is called the monodromy group and denoted $\Mon^2(X)$.

\begin{defn} {\rm (\cite[Def. 1.2]{Mo})}\label{def:wall}
Let $X$ be an IHS  manifold and let $D$ be a divisor on $X$. Then $D$ is called a wall divisor if $q(D)<0$ and $f(D)^\perp\cap\mathcal{BK}_X=\emptyset$ for all Hodge isometries $f\in \Mon^2(X)$. The set of wall divisors on $X$ is denoted by $W_X$. 
\end{defn}

The ample cone is one of the connected components of $\mathcal{C}_X-\cup_{D\in W_X}D^\perp$. 

Wall divisors are closely related to extremal rays of the Mori cone, as was analised independently in \cite{BHT} and \cite{Mo}. In particular, dual divisors to generators of rational extremal rays of negative square are wall divisors by \cite[Lemma 1.4]{Mo}. Notice that the extremal rays needed to determine the \kahl cone are indeed rational since the part of the Mori cone of curves of negative square is locally a finite rational polyhedron \cite[Cor. 18]{HT2}. The analogy runs deeper:

\begin{prop}\label{prop:wall_vs_mbm} 
Let $D$ be a divisor and let $R$ be the primitive class $D/\div(D)\in H_2(X,\mathbb{Z})\subset H^2(X,\mathbb{Q})$. Then $D$ is a wall divisor if and only if there exists a Hodge isometry $f\in \Mon^2(X)$ such that $f(R)$ generates an extremal ray of the Mori cone on some IHS  manifold $X'$ birational to $X$.
\end{prop} 

\begin{proof}
Let $D$ be a wall divisor. As $q(D)<0$, we have $D^\perp\cap \mathcal{C}_X\neq\emptyset$. Therefore, if $X$ is projective, there is a Hodge isometry $f\in \Mon^2(X)$ such that $f(D)^\perp\cap \overline{\mathcal{BK}}_X\neq\emptyset$  by \cite[Thm. 6.18 (2)]{mark_tor}.  If $X$ is not projective, the same result is a direct consequence of \cite[Cor. 5.2 and Rem. 5.4]{H1}, where the cycle $\Gamma$ in the mentioned results is of parallel transport and acts as a Hodge isometry on $H^2(X,\mathbb{Z})$.  By definition of wall divisor, $f(D)^{\perp} $ supports a component of the boundary of $\mathcal{BK}_X$. Up to taking a different birational model $X'$ of $X$, we can suppose $f(D)^{\perp}\cap \overline{\mathcal{K}}_X\neq\emptyset$. As the ample cone is locally rationally polyhedral by \cite[Prop. 13]{HT2}, we can also suppose that $f(D)^\perp$ supports a face of this cone (again, if needed, by changing birational model). This implies that $R$ is an extremal ray. 

The converse is the content of \cite[Lemma 1.4]{Mo} (see also \cite[Prop. 3]{BHT}).  
\end{proof}

\begin{remark} \label{rem:MBM}
The above result is also implied by \cite[Cor. 6]{BHT} and can be used in order to give an equivalent definition of wall divisors, i.e., divisors dual to extremal rays up to the action of parallel transport Hodge isometries. In other words, the MBM classes defined in \cite{AV} are exactly the classes of curves dual to wall divisors.
\end{remark}

A different characterisation of wall divisors can be given in terms of contractions:

\begin{thm} \label{thm:mazzone}
Let $R$ be a primitive rational curve on a projective IHS  manifold $X$ such that the dual divisor $D$ is a wall divisor. Then one the following cases occurs:
\begin{itemize}
\item[(i)] $\mathrm{Locus}(\mathbb{R}^+[R])$ contains a divisor of class a multiple of $D$. Furthermore, there exists a birational map $f\,:\,X\dashrightarrow Y$ with $Y$ singular symplectic such that $f$ contracts $R$.
\item[(ii)]  For a general small deformation $(X_t,R_t)$ of $(X,R)$ the locus $\mathrm{Locus}(\mathbb{R}^+[R_t])$ is not a divisor and there exists an IHS  manifold $X'_t$ along with a birational map $f_t\,:\,X_t\dashrightarrow X'_t$ and a morphism $X'_t\rightarrow Y_t$ contracting $f_t(R_t)$.
\end{itemize}
\end{thm}

\begin{proof}
Let $X''$ be an IHS  manifold deformation of $X$ such that the parallel transport $R''$ of $R$ is an effective rational curve generating the algebraic classes of $H_2(X'',\mathbb{Z})$ (cf. \cite[Thm. 1.3]{Mo} for the existence of such an $X''$). Let $D''$ be the dual divisor to $R''$. 

Suppose that $\mathrm{Locus}(\mathbb{R}^+[R''])$ has codimension one (thus, the same holds for $\mathrm{Locus}(\mathbb{R}^+[R])$ by semicontinuity) and let $bD''$ be the class of its closure.
As we deform back to $X$, the divisor $bD''$ deforms to $bD$, which is thus effective and is contained in $\mathrm{Locus}(\mathbb{R}^+[R])$. As $D\cdot R<0$, the MMP for the pair $(X,D)$ yields the existence of a birational map $f$ as in item (i).

Let us suppose now that $\mathrm{Locus}(\mathbb{R}^+[R''])$ has codimension at least two and show that we fall in case (ii). Under this assumption $X''$ contains no effective divisor. Then, by the wall and chamber decomposition of the positive cone given in \cite[\S 5]{mark_tor}, the closure of the birational K{\"a}hler cone of $X''$ coincides with its positive cone. On the other hand, as the curve $R''$ is effective, the K{\"a}hler cone is the intersection of the positive cone with the half space of real $(1,1)$-classes intersecting $R''$ positively. By the definition of the birational K{\"a}hler cone, this yields the existence of an IHS  manifold $Z''$ along with a birational map $X''\dashrightarrow Z''$, the indeterminacy locus of which is $\mathrm{Locus}(\mathbb{R}^+[R''])$. In particular, the class $-R''$ is effective on $Z''$ as proved in \cite[Cor. 2.4]{huy_kahl}. We now deform $X''$ (hence, also $X$) to a projective IHS manifold where the class of $R''$ is still effective; this is possible as, by  \cite[Prop. 3]{BHT}, all small deformations of $X$ where $D$ stays of type $(1,1)$ have $R$ or $-R$ effective and projective deformations are dense. In particular, we can choose a projective deformation $X'''$ where the parallel transport of $R$ is effective and extremal; indeed, up to changing birational model, $R$ is an extremal ray on all deformations $(X_0,R_0)$ belonging to the Zariski open set where $\overline{\mathcal{C}_{X_0}}=\overline{\mathcal{BK}_{X_0}}$. Therefore, the Contraction Theorem  
yields a contraction $X'''\rightarrow Y'''$ and the conclusion follows from the next lemma.
\end{proof}

\begin{lemma} \label{lemma:cago} 
Let $Z$ be a projective IHS  manifold and let $R$ be a curve generating an extremal ray such that $\mathrm{Locus}(\mathbb{R}^+[R])$ has codimension at least $2$. Let $Z \rightarrow Y$ be the contraction of this extremal ray. Then for all small locally trivial deformations $Y_t$ of $Y$ there is a symplectic resolution  $Z_t\rightarrow Y_t$ contracting exactly  $\mathrm{Locus}(\mathbb{R}^+[R_t])$, where $(Z_t,R_t)$ is a small deformation of $(Z,R)$.
\end{lemma}

\begin{proof}
By \cite[Thm. 1.3]{Wie}, the singular locus of $Y$ has codimension at least four. 
Let $Y_t$ be a locally trivial small deformation of $Y$. Then $Y_t$ has the same Beauville-Bogomolov form of that of $Y$ (and also the same second Betti number) and it has a symplectic resolution $Z_t$, which is a small deformation of $Z$ by Theorem \ref{thm:nami}. Remark \ref{oss:loc_triv} ensures that the deformation $[R_t]$ of $[R]$ is algebraic. As $R$ is extremal, small deformations $[R_t]$ of its class are represented by curves $R_t$ \cite[Prop. 3]{BHT};  the Rigidity Lemma then implies that $R_t$ is contracted by $Z_t\to Y_t$.
 By Remark \ref{oss:loc_triv}, $b_2(Z)=b_2(Y)+1$. Hence, $b_2(Z_t)=b_2(Y_t)+1$ and the map contracts precisely $\mathrm{Locus}(\mathbb{R}^+[R_t])$.
\end{proof}

\begin{remark} \label{rem:eccezziunali}
The first item of Theorem \ref{thm:mazzone} is slightly stronger than \cite[Prop. 6.1]{mark_tor} as it ensures that {\em exceptional divisors}, as defined in \cite[Def. 5.1]{mark_tor}, are contractible, up to birational equivalence. This should be regarded as the higher dimensional analogue of the contractability of  effective divisors with self-intersection $-2$ on $K3$ surfaces. Notice that, when $R$ is reducible, the contraction does not necessarily have relative Picard rank one. The contraction map $f\,:\,X\dashrightarrow Y$ is a composition of flops and divisorial contractions and therefore is only rational. The second item of the proposition cannot be strengthened and in particular it might not hold for $(X,R)$. Indeed, one has to take into account the action of the subgroup $W_{exc}$ of $\Mon^2$ generated by the reflections on reduced and irreducible exceptional divisors. The general deformations in the statement are precisely those manifolds where $W_{exc}$ is the identity. Note that this  set strictly contains the open set of manifolds with an irreducible Hodge structure and it is Zariski open as the set of generators of $W_{exc}$ is finite up to the monodromy action.
\end{remark}

Wall divisors on $S^{[k]}_\eps$ can be determined lattice-theoretically using results of Yoshioka \cite{yoshi} and Bayer and Macr\`i \cite{bm}. In the following, we use the same notation as in Remark \ref{rem:hodge}.

\begin{remark}\label{bridg}
In \cite{bm} and \cite{yoshi}, Bayer, Macr\`i and Yoshioka determine a decomposition of the space of stability conditions $\mathrm{Stab}^0(S,v)$ given by walls and chambers. Any stability condition $\sigma$ in a chamber gives a smooth moduli space $M(v,S,\sigma)$ of stable objects in $D^b(S)$ with Mukai vector $v$, whereas any condition lying on a wall gives a singular space and conditions on nearby chambers give its symplectic resolution. Moreover, for every $\sigma$ in a chamber of $\mathrm{Stab}^0(S,v)$, \cite[Thm 1.2]{bm} gives a map from $\mathrm{Stab}^0(S,v)$ to the positive part of the movable cone $\overline{\mathcal{BK}_{M(v,S,\sigma)}}$, and every chamber lands in $\mathcal{BK}_{M(v,S,\sigma)}$. By Proposition \ref{prop:wall_vs_mbm}, this implies that all walls of $\mathrm{Stab}^0(S,v)$ are dual to wall divisors and, up to the action of $W_{exc}$ (defined in Remark \ref{rem:eccezziunali}), we obtain all wall divisors of $M(v,S,\sigma)$ in this way. By Remark \ref{rem:hodge} along with the fact that Mumford's stability lies in  $\mathrm{Stab}^0(S,v)$ for any $v$, the ordinary moduli spaces of Mumford's stable sheaves with Mukai vector $v$ is obtained as  $M(v,S,\sigma)$ for a $\sigma\in \mathrm{Stab}^0(S,v)$. In particular, $S^{[k]}_\eps$ is the Albanese fibre of some $M(v,S,\sigma)$.
\end{remark}

\begin{thm} 
\label{thm:muri}
Let $D$ be a divisor of $S^{[k]}_\eps$ with $q(D)<0$ and let $T\subset \Lambda:=H^{2*}(S,\mathbb{Z})$ be the saturated lattice generated by $v:=(1,0,1-2\eps-k)$ and $D$. Then $D$ is a wall divisor if and only if there is an $s\in T$ such that

\begin{itemize}
\item[(i)]  $0\leq q(s)<b(s,v)\leq (q(v)+q(s))/2$; or,
\item[(ii)] $\eps=0$, $q(s)=-2$ and $0\leq b(s,v)\leq q(v)/2$.
\end{itemize}
\end{thm}

\begin{proof}
Remark \ref{bridg} implies that all wall divisors of $S^{[k]}_\eps$ correspond to walls in the space $\mathrm{Stab}^0(S,v)$.

For $\eps=0$ we can thus apply \cite[Thms. 5.7 and 12.1]{bm} with $a:=s$ and $b:=v-s$; our inequalities are equivalent to imposing that both $a$ and $b$ are in the {\em positive cone of $T$} (cf. \cite[Def. 5.4]{bm}), i.e., $q(a)\geq 0$ and $b(v,a)>0$ and the same for $b$. 

For $\eps=1$ the statement follows from\cite[Prop. 1.3]{yoshi}. Indeed, the conditions in \cite[Def. 1.2]{yoshi} can be rephrased by asking that $a:=s$ and $b:=v-s$ are in the positive cone of $T$ as before. The additional condition $b(s,v)^2 > q(v)q(s)$ in \cite[Prop. 1.3]{yoshi} is equivalent to the requirement that $T$ is indefinite, which is implied by $q(D)<0$. 
\end{proof}

\begin{remark}\label{rem:nonprim}
A lattice $T$ as in the above theorem can contain several elements $s$ satisfying (i) and (ii), and abstractly isometric lattices can even correspond to different kinds of wall divisors, as the following example illustrates (cf. also \cite[Sec. 4]{HT_new}). Let $k-1+2\eps=2rt$, where $r$ and $t$ are relatively prime integers. Let $S$ be a symplectic surface and let $\mathcal{M}$ be the moduli space of stable sheaves with Mukai vector $v:=(r,0,-t)$. Let $\Gamma\in H^{1,1}(\mathcal{M},\ZZ)$ be the  image of $(r,0,t)$ under the natural Hodge isometry $H^2(\mathcal{M},\ZZ)\cong v^\perp\subset H^{2*}(S,\ZZ)$. The saturated lattice generated by $v$ and $\Gamma$ is isometric to $U$ and contains no elements $s$ such that $q(s)=0$ and $b(s,v)=1$, unless either $r$ or $t$ are $1$. Note that $\frac{v+\Gamma}{2r}$ and $\frac{v-\Gamma}{2t}$ satisfy the conditions of the above theorem, and hence $\Gamma$ is a wall divisor. The lattice $U$ is also associated with the exceptional divisor $\Delta_k$ of $S^{[k]}_\eps$, but in the saturated lattice generated by $v$ and $\mathfrak{e}_k$ there is an element $s$ such that $b(s,v)=1$ and $q(s)=0$. However, isometric lattices as in Theorem \ref{thm:muri} give rise to isometric wall divisors.
\end{remark}

Theorem \ref{thm:muri} enables us to extend  to manifolds of Kummer type a result obtained by Bayer, Hassett and Tschinkel, and independently by the third author, in the case of manifolds of $K3^{[k]}$ type.

\begin{prop} \label{prop:bound} 
Let $R$ be a primitive generator of an extremal ray of the Mori cone of a manifold $X$ deformation of $S^{[k]}_\eps$. Then $q(R)\geq -(k+3-2\eps)/2$.
\end{prop}
\begin{proof}
For $\eps=0$ this is the content of \cite[Cor. 2.7]{Mo} or \cite[Prop. 2]{BHT}.

Let $\eps=1$ and $q(R) <0$. Then the dual divisor $D$ to $R$, namely, $R=D/\div(D)$, is a wall divisor by 
Proposition \ref{prop:wall_vs_mbm}. As wall divisors are invariant under deformation, we can assume \mbox{$X=S^{[k]}_1$} for some abelian surface $S$. Let $T,v,s$ be as in Theorem \ref{thm:muri}. Let $a:=\GCD(q(v),b(s,v)).$ We have $aD=b(s,v)v-q(v)s$ and $\div(D)=q(v)/a$. 
Then we have	
\begin{eqnarray*}
q(D)&=&(q(v)^2q(s)-q(v)b(s,v)^2)/a^2\geq\\
&\geq& \frac{4q(v)^2q(s)-q(v)^3-q(v)q(s)^2-2q(v)^2q(s)}{4a^2} \geq -\frac{q(v)^3}{4a^2}=-\frac{(k+1)\div(D)^2}{2},
\end{eqnarray*} 
where we have used the inequality  $b(s,v)\leq (q(v)+q(s))/2$.
\end{proof}

The above statement in the $K3$ case is part of a conjecture by Hassett and Tschinkel \cite[Conj. 1.2]{HTint}, who predicted that the class $R$ of a primitive $1$-cycle in a manifold of $K3^{[k]}$-type is effective if and only if the inequality in Proposition \ref{prop:bound} holds. Counterexamples to the if part are known, cf. \cite[Rem. 10.4]{bm2} and \cite[Rem. 8.10]{ck}. The analogous conjecture for manifolds $X$ of Kummer type was stated only in the four-dimensional case \cite[Conj. 1.4]{HTint}. Proposition \ref{prop:bound} shows that the only if part holds independently of the dimension of $X$; on the other hand, the if part fails as soon as $\dim X>4$, as the following example shows.

\begin{example}\label{ex:curvaza}
Let $S$ be an abelian surface with an order four symplectic group automorphism $\varphi$. Such an automorphism induces an automorphism $\varphi$ of order four on all the generalised Kummer manifolds arising from $S$. There exists a primitive non-effective class $F\subset \NS(S)$ such that \mbox{$\varphi(F)=-F$} and $F^2=-2$, cf. \cite[Table 15]{fuj}. This class gives a $1$-cycle class in $N_1(S_1^{[k]})$ that is orthogonal to any $\varphi$-invariant ample class (hence, it is not effective) and has square $-2$. This shows that the inequality in Proposition \ref{prop:bound} is not sufficient for the effectivity of a $1$-cycle. \end{example}

 We now state a criterion for determining whether a projective manifold of $K3^{[k]}$ or Kummer type is isomorphic to $S^{[k]}_\eps$ for some $S$.

\begin{prop} \label{prop:qualisono}
Let $X$ be a projective manifold of $K3^{[k]}$ or Kummer type. Then $X$ is isomorphic to $S^{[k]}_\epsilon$ for some $S$ if and only if  there is a birational map $f:S^{[k]}_\eps\dashrightarrow X$ and $f^*[D] \in \mathfrak{e}_k^\perp$ for some nef divisor $D\in \NS(X)$. 
\end{prop}

\begin{proof}
The only if part is trivial and we prove the converse implication.

We first claim that $\overline{\mathcal{BK}}_{S^{[k]}_\epsilon}\cap \mathfrak{e}_k^\perp=\overline{\mathcal{K}}_{S^{[k]}_\epsilon}\cap \mathfrak{e}_k^\perp $, that is, all movable divisors on $\mathfrak{e}_k^\perp$ are nef.  Granting this, the divisor class $f^*[D]\in  \NS(S^{[k]}_\epsilon)$ lies in the image of \eqref{k_immersion} and is movable, hence nef. Moreover, the pullback under $f$ of a small ample modification of $D$ is ample on $S^{[k]}_\eps$ and thus $X\simeq S^{[k]}_\eps$ by the global Torelli Theorem \cite[Thms. 1.2 and 1.3]{mark_tor}.

It remains to prove the claim. Let $E\in\overline{\mathcal{C}_{S^{[k]}_\epsilon}}$ be a divisor such that $b(E,\mathfrak{e}_k)=0$. In particular, the class $[E]$ lies in the image of the restriction of \eqref{k_immersion} to the closure of the positive cone $\overline{\mathcal{C}_S}$ and we will denote by $E_S$ an effective divisor on $S$ representing its preimage. Let us assume that $[E]$ is not nef.  
Any irreducible curve $\Gamma\subset S^{[k]}_\epsilon$ such that $\Gamma \cdot E<0$ is not contained in $\Delta_k$. The image of such a $\Gamma$ under the projection to $S$ of the incidence variety 
\begin{equation}\label{eq:incidenzafiga}
I:=\{(P,[Z])\in S\times S^{[k]}_\epsilon\,|\, P\in\Supp(Z)\}
\end{equation}
is an effective curve $\Gamma_S\subset S$, whose class is sent to $[\Gamma]$ by \eqref{imm2}. Since $E_S\cdot \Gamma_S<0$, the divisor $E_S$ is not nef. In the abelian case this is impossible and hence $[E]$ is nef and we are done.  Let us show that in the $K3$ case $[E]$ is not movable. Let $R\subset S$ be a ($-2$)-curve such that $E_S\cdot R<0$ and denote by $D_R\subset S^{[k]}$ the corresponding uniruled divisor defined as in \eqref{imm3}. Then $b(E,D_R)<0$, whence $E$ is not movable by \cite[Prop. 5.6]{mark_tor}. 
 \end{proof}
 
 \begin{remark}
In the above proposition the condition that $X$ is birational to $S^{[k]}_\eps$ is equivalent to asking that there is a parallel transport Hodge isometry between the two manifolds, cf. \cite[Thm. 1.3]{mark_tor}. If $S$ is $K3$, there is a topological way of recognizing a parallel transport Hodge isometry, cf. \cite[Cor. 9.5]{mark_tor}. By the computation of the monodromy group in the Kummer case \cite[Thm. 2.3]{mon_mon}, it is highly expected that a similar characterisation holds if $S$ is abelian. 
 \end{remark}

We end this section with a result that will be used in the proof of Theorem \ref{uno}. 

\begin{prop}\label{bundle_finito}
Let $X$ be a holomorphic symplectic manifold, i.~e., there is an \'{e}tale cover \mbox{$\widetilde{X}:=\Pi_{i\in I} M_i\to X$,} where every $M_i$ is either IHS or abelian. For every subset $J\subset I$, denote by $F_J$ the image in $X$ of a general fibre of the projection $\widetilde{X}\to \Pi_{j\in J} M_j$. Let
$q:\,\mathcal{P}\rightarrow X$ be generically a $\mathbb{P}^r$-bundle. 

Assume that
$g\,:\,\mathcal{P}\dashrightarrow Y$ is a rational map to an IHS manifold $Y$
such that:
\begin{itemize}
\item $\dim Y=2r+\dim X$;
\item  $g$ is well-defined in codimension one;
\item $g$ is injective on general fibres of $q$;
\item for all $J$, the map $g$ is generically injective when restricted
to $\mathcal{P}_{|F_J}$;
\item the image of $g$ is an irreducible component of the locus covered by the rational curves of class $[g(\ell)]$, where $\ell$ is a line in a fibre of $q$.
\end{itemize}

Then $g$ is finite.
\end{prop}
\begin{proof}
Let $T$ denote the closure of the image of $g$ and $h:\widetilde{T}\to
T$ be its desingularization.  We consider the maximal rationally
connected fibration
$\pi\,:\,\widetilde{T}\dashrightarrow B$ of $T$.   We denote by
$\tilde{g}:\mathcal{P}\dashrightarrow \widetilde{T}$ the rational map induced
by $g$ and assume that a general fibre of $g$ (or, equivalently, of
$\tilde{g}$) has dimension $\alpha$. By
\cite[Thm. 4.4]{AV} along with the equality $\dim Y=2r+\dim X$, a
general fibre $F$ of $\pi$ has dimension equal to $\codim_X
T=r+\alpha$ and
$\tilde{g}^{-1}(F)$ has dimension $r+2\alpha$. As $g$ is injective on a general
fibre of $q$, the locus $q(\tilde{g}^{-1}(F))$ is $2\alpha$-dimensional. 

Let $\sigma$ be a symplectic form on $Y$. As in \cite[Pf. of Thm.
4.4]{AV}, one shows that the form $h^*(\sigma|_T)$ is
degenerate precisely on the fibres of $\pi$, which are rationally
connected and hence have no $2$-forms. By definition of $\tilde{g}$, the
$2$-form $g^*(\sigma_{|T})$ coincides with $\widetilde{g}^*(h^*(\sigma_{|T}))$ where the latter is defined. Since $g^*(\sigma_{|T})$ is well-defined in
codimension one,  it extends to a $2$-form on $\mathcal{P}$ that is
degenerate along $\widetilde{g}^{-1}(F)$. On the other hand, any form on $\mathcal{P}$ is the pullback of a form on $X$ and forms on $X$ can be degenerate only along the $F_J$'s. Therefore, if $\alpha>0$, then the closure of $q(\tilde{g}^{-1}(F))$ coincides with $F_J$ for some $J\subset I$. This contradicts the injectivity of the restriction of $g$ to $\mathcal{P}_{|F_J}$.
\end{proof}

\section{Curves on symplectic surfaces and their pencils}\label{sec:quattro}

For a polarized surface $(S,L)$, we denote by 
$\{L\}$ the continuous system of $L$, that is, the connected component of $\Hilb(S)$ containing the linear system $|L|$. If $S$ is a $K3$ surface, then $|L|=\{L\}$. If $S$ is an abelian surface, then $\{L\}$ is obtained translating curves in $|L|$ by points of $S$. We denote by $V_{|L|,\delta}(S)$ (respectively, $V_{\{L\},\delta}(S)$) the Severi variety of $\delta$-nodal curves in $\{L\}$ (resp. $|L|$), and by $\{L\}^1_{\delta,d}$  (resp. $|L|^1_{\delta,d}$) the Brill-Noether locus parametrizing the nodal curves whose normalization carries a $g^1_d$. We recall \eqref{eq:defeps} and the following result.

\begin{thm} \label{thm:exist}
 Let $(S,L)$ be a general  polarized $K3$ or abelian surface of genus $p:=p_a(L)$. Let $\delta$ and $k$ be integers satisfying $0 \leq \delta \leq p-2\eps$ and $k+\eps \geq 2$.  Then the following hold:
  \begin{itemize}
\item [(i)] $\{L\}^1_{\delta,k+\eps} \neq \emptyset$ if and only if 
\begin{equation} \label{eq:boundA}
\delta \geq \alpha\Big(p-\delta-\eps-(k-1+2\eps)(\alpha+1)\Big), 
\end{equation}
where
\begin{equation} \label{eq:alpha}
\alpha= \Big\lfloor \frac{p-\delta-\eps}{2(k-1+2\eps)}\Big\rfloor; 
\end{equation}
\item [(ii)]  whenever nonempty, $\{L\}^1_{\delta,k+\eps}$ is equidimensional of dimension  $\min\{p-\delta,2(k-1+\eps)\}$ and a  general element in each component is an irreducible curve $C$ with normalization $\widetilde{C}$ of genus $g:=p-\delta$ such that $\dim G^1_{k+\eps}(\widetilde{C})=\max\{0,\rho(g,1,k+\eps)=2(k-1+\eps)-g\}$;
\item [(iii)]  there is at least one component $Y_{\delta,k+\eps}$ of $\{L\}^1_{\delta,k+\eps}$ where, for $C$ and $\widetilde{C}$ as in (ii), when $g \geq 2(k-1+\eps)$ (respectively $g < 2(k-1+\eps)$), any (resp. a general) $g^1_{k+\eps}$ on $\widetilde{C}$ has simple ramification and all nodes of $C$ are non-neutral with respect to it.
Furthermore, when $S$ is abelian, for general $C$ in this component the Brill-Noether variety $G^1_{k+1}(\widetilde{C})$ is reduced. 
\end{itemize}
\end{thm}

\begin{proof}
  This is \cite[Thm. 1.6]{klm} when $S$ is abelian and \cite[Thm. 0.1]{ck}, combined with \cite[Rem. 5.6]{klm}, when $S$ is $K3$.
\end{proof}

\begin{remark} \label{rem:boundequiv}
(i)  The condition \eqref{eq:boundA} is equivalent to
 \begin{equation}
    \label{eq:boundB}
\rho(p,l,(k+\eps)l+\delta)+\eps l(l+2) \geq 0 \; \; \mbox{for all integers $l \geq 0$}.   
  \end{equation}
Indeed, the left hand side of \eqref{eq:boundB} attains its minimum for $l=\alpha$ as in \eqref{eq:alpha} and \eqref{eq:boundA} is a rewrite of  \eqref{eq:boundB} with $l=\alpha$.
(ii) The condition \eqref{eq:boundA} is also {\it necessary} for the existence of a curve in $\{L\}$ with {\it partial normalization} of arithmetic genus $g:=p-\delta$ carrying a $g^1_{k+\eps}$. This follows from \cite[Thm. 5.9 and Rem. 5.11]{klm} in the abelian case and \cite[Thm. 3.1]{ck} in the $K3$ case, by remarking that the proofs go through replacing the normalization of the curve with a partial normalization, as remarked in \cite[Rem. 3.2(b)]{ck}. 
\end{remark}

Let $\gg$ be a linear series of type $g^r_{k+\eps}$ on the normalization $\widetilde{C}$ of a curve $C\subset S$, that is, $\gg=(A,V)$, where $A$ is a line bundle of degree $k+\eps$ and $V \sub H^0(A)$ is an $(r+1)$-dimensional subspace.
If $\gg$ is base point free, we have a natural rational map
\begin{equation}\label{hilb_map}
\iota_{\gg}:\mathbb{P}^r:=\mathbb{P}(V) \dashrightarrow  S^{[k+\eps]}
\end{equation}
obtained from the composition $\PP(V) \sub |A|\subset\Sym^{k+\eps}(\widetilde{C}) \rightarrow \Sym^{k+\eps}(C) \subset \Sym^{k+\eps}(S)$, whose image does not lie in the exceptional locus $\Delta_k$ of the Hilbert-Chow morphism. Thus, $\gg$ defines a rational $r$-fold inside the Hilbert scheme $S^{[k+\eps]}$. In particular, when $r=1$, we obtain a rational curve.

When $S$ is an abelian surface, the Albanese map $\Sigma_k$ restricted to the image of $\iota_{\gg}$ is constant, because otherwise we would get a rational curve in $S$. Therefore, up to translating the curve $C$, we may assume that \eqref{hilb_map} lands into the generalised Kummer variety $K^k(S)$. Let us now specialise to the case $r=1$, denote by $\nu:\widetilde{C} \to C$ the normalization, and let $R_{C,\nu_*\gg} \subset  S_{\eps}^{[k]}$ be the rational curve image of $\iota_{\gg}$, recalling \eqref{eq:defs}. (The same construction can be performed for any linear series on a partial normalization of $C$.)

\begin{lemma}\label{juveladra}
Let  $C \in \{L\}^1_{\delta,k+\eps}$ be a curve whose normalization  possesses a linear series $\gg$ of type $g^1_{k+\eps}$ with simple ramification and such that all nodes of $C$ are non-neutral with respect to it.

Then the class of the rational curve $R_{C,\nu_*\gg}$ in $H_2(S_{\eps}^{[k]},\mathbb{Z})$  with respect to the decomposition \eqref{eq:N1} is
\begin{equation}\label{class}
R_{p,\delta,k}:=L-(p-\delta+k-1+\eps)\mathfrak{r}_k
\end{equation}
and its dual divisor class is
\begin{equation}\label{classdual}
D_{p,\delta,k}:=L-\frac{(p-\delta+k-1+\eps)}{2(k-1+2\eps)}\mathfrak{e}_k,
\end{equation}
\end{lemma}

\begin{proof}
In the $K3$ case, this is \cite[Lemma 2.1]{ck}. The proof in the abelian case is similar. 
\end{proof}

In particular, one has (with $\alpha$ as in \eqref{eq:alpha}): 
\begin{equation} \label{eq:squarecurve}
q(R_{p,\delta,k})= 2(p-1)-\frac{(p-\delta+k-1+\eps)^2}{2(k-1+2\eps)}.
\end{equation}
Observe that, for fixed values of $k$ and $p$, the minimum in \eqref{eq:squarecurve}, as well as the maximal ``slope'' $p-\delta+k-1+\eps$ of the class $R_{p,\delta,k}$, is reached for a curve with the minimal number of nodes.

\begin{remark} \label{rem:cile1} 
Under the same hypotheses as in Lemma \ref{juveladra}, one may rewrite \eqref{class} as 
\[
q(R_{p,\delta,k})= 2\Big(\rho+\eps\alpha(\alpha+2)+ \eps-1\Big)-\frac{\beta^2}{2(k-1+2\eps)},
\]
with
\[ \rho:=\rho(p,\alpha,(k+\eps)\alpha+\delta) \; \; \mbox{and} \; \;
  \beta:= (2\alpha+1)(k-1+2\eps) -p+\delta+\eps.
\]
In particular, we have $-(k-1+2\eps) < \beta \leq k-1+2\eps$, and \eqref{eq:boundA}, or equivalently \eqref{eq:boundB} with $l=\alpha$, says that $\rho+\eps\alpha(\alpha+2) \geq 0$. From these inequalities one reobtains the bound from Proposition \ref{prop:bound}:
\[ q(R_{p,\delta,k}) \geq -\frac{k+3-2\eps}{2}, \]
with equality if and only if
\[ p=\alpha(\alpha+1)(k-1+2\eps)+\eps \; \; \mbox{and} \; \;  \delta=\alpha(\alpha-1)(k-1+2\eps)\]
(see also \cite[Cor. 3.4]{ck}).
\end{remark}

Proposition \ref{prop:bound} yields the following extension of \cite[Cor. 8.6]{ck} to Kummer manifolds.

\begin{cor}\label{cor:cile2} 
 Assume that $\NS(S) \cong \mathbb{Z}[L]$. Let $n\in \mathbb{Z}^{>0}$ and set  
$p:=n(n+1)(k-1+2\eps)+\eps$ and $\delta:=n(n-1)(k-1+2\eps)$.
Then the rational curves in $S_{\eps}^{[k]}$ obtained from the component $Y_{\delta,k}$ of Theorem \ref{thm:exist} generate extremal rays of $S_{\eps}^{[k]}$.
\end{cor}

We also have:

\begin{proposition} \label{prop:andreas010715}
  The rational curves  in $S_{\eps}^{[k]}$ obtained from any component of the relative Brill-Noether variety $\G^1_k(\{L\}^1_{\delta,k+\eps})$ parametrizing pairs $(C,\gg)$ such that $C\in \{L\}^1_{\delta,k+\eps}$ and $\gg$ is a $g^1_{k+\eps}$ on the normalization of $C$ move in a family of rational curves of dimension precisely $2k-2$.

Any small deformation $X_t$ of $X_0=S_{\eps}^{[k]}$ keeping the class of the rational curves algebraic  contains a $(2k-2)$-dimensional family of rational curves that are deformations of the rational curves in $S_{\eps}^{[k]}$. 
\end{proposition}

\begin{proof}
Any irreducible family of rational curves in $S_{\eps}^{[k]}$ containing our family yields, by the incidence \eqref{eq:incidenzafiga}, a family of pairs $(C,\gg)$ with $C \in \{L\}$ and $\gg$ a linear series of type $g^1_{k+\eps}$ on the normalization of $C$. By 
\cite[Thm. 5.3]{klm}, the rational curves will therefore move in a family of dimension precisely $2k-2$, which is the expected dimension of any family of rational curves on a $(2k)$-dimensional IHS manifold  \cite[Cor. 5.1]{ran}. Hence, 
as a consequence of \cite[Cor. 3.2-3.3]{ran}, the rational curves will deform to any $X_t$ as in the statement, cf. \cite[Pf. of Prop. 3]{BHT}. 
\end{proof}

\section{Examples of wall divisors}\label{sec:cinque}

Let $(S,L)$ be a general abelian or $K3$ surface, and  fix integers $p$, $k$ and $\delta$ satisfying \eqref{eq:boundA}. Let $R_{p,\delta,k}$ be as in \eqref{class} and denote by $D_{p,\delta,k}$ its dual (class) divisor.

\begin{thm}\label{thm:muri_curve_doppio}
The divisor $D_{p,\delta,k}$ is a wall divisor if and only if $q(R_{p,\delta,k})<0$.

\end{thm}

\begin{proof}
By Proposition \ref{prop:andreas010715}, the family of rational curves with class $R_{p,\delta,k}$ has a component of dimension $2k-2$ and deforms in all small deformations $X_t$ of $S^{[k]}_\eps$ where the class $R_{p,\delta,k}$ remains algebraic. Let $(X_t,R_t)$ be a very general such deformation. The class $R_t$ spans $N_1(X_t)$, hence it is extremal. As it has negative square, its dual is a wall divisor. Since wall divisors are invariant under deformation, $D_{p,\delta,k}$ is a wall divisor on $S^{[k]}_\eps$, too. 
\end{proof}

\begin{remark} \label{rem:cacca_doppio}
If $D_{p,\delta,k}$ is a wall divisor, we can recover the lattice $T$ associated with it in Theorem \ref{thm:muri}. Set $a:=\GCD(2k-2+4\eps,g+k-1+\eps)$, $ab:=g+k-1+\eps$ and $ac:=2k-2+4\eps$. The saturation of the lattice generated by $v$ and $D_{p,\delta,k}$ is $T:=\langle v,w\rangle$, where $w=\frac{b}{c}(v-\mathfrak{e}_k)+L-v$. Note that $q(w)=2\delta-2+2\eps$ and $b(w,v)=g-k+1-3\eps$. The element $w$ does not necessarily satisfy the inequalities (i) or (ii) in Theorem \ref{thm:muri} for $s$. However, this occurs in some special cases, e.g., in the examples below. 
\end{remark}

\begin{example}\label{mitico}
Let $p=2k-2+5\eps$ and $\delta=0$. Then $q(R_{p,\delta,k})=-\frac{k+3-2\eps}{2}$ and the lattice $T$ associated with $R_{p,\delta,k}$ is isometric to 
\begin{footnotesize} $\left( \begin{array}{cc} -2+2\eps & k-1+2\eps\\ k-1+2\eps & 2k-2+4\eps \end{array}\right)$\end{footnotesize}, cf. Remark \ref{rem:cacca_doppio}.
\end{example}
 
\begin{example} 
Let $p=2k-2+5\eps-a$, $a\leq k-1+2\eps$, and $\delta=0$. Then $q(R_{p,\delta,k})<0$ and the lattice $T$ associated with $R_{p,\delta,k}$ is isometric to \begin{footnotesize} $\left( \begin{array}{cc} -2+2\eps & k-1+2\eps-a\\ k-1+2\eps-a & 2k-2+4\eps \end{array}\right)$\end{footnotesize}, cf. Remark \ref{rem:cacca_doppio}.  
\end{example}

\begin{example}
Let $p=2k-2+5\eps$ and $0 \leq \delta \leq \frac{k-1+2\eps}{2}$. Then $q(R_{p,\delta,k})<0$ and the lattice $T$ associated with $R_{p,\delta,k}$ is isometric to \begin{footnotesize} $\left( \begin{array}{cc} 2\delta-2+2\eps & k-1+2\eps\\ k-1+2\eps & 2k-2+4\eps \end{array}\right)$\end{footnotesize}, cf. Remark \ref{rem:cacca_doppio}.  
\end{example}

\begin{prop}\label{prop:tutti_muri}
Let $k\geq 2$ be an integer and set $\eps=0$ (respectively, $\eps=1$). Let $v:=(1,0,1-2\eps-k)$ and let $s\in\Lambda=U^{\oplus 4}\oplus E_8(-1)^{\oplus2-2\eps}$ be an element satisfying the inequalities (i) or (ii) in Theorem \ref{thm:muri}. Let $T=\langle v,s\rangle$. 
Then there exists a primitively polarized $K3$ (resp. abelian) surface $(S,L)$ of genus $p$ and an integer $0\leq \delta\leq p-2\eps$ such that $p,\delta, k$ satisfy \eqref{eq:boundA} and the following hold:
\begin{itemize}
\item[(a)]the divisor $D_{p,\delta,k}$ is a wall divisor;
\item[(b)] the saturation of the lattice generated by $v$ and $D_{p,\delta,k}$ in $\Lambda$ is isometric to $T$.
\end{itemize}    
\end{prop}

\begin{proof}
By Theorem \ref{thm:exist}(i), as soon as $\{L\}^1_{\delta,k+\eps}$ is non-empty, then $\{L\}^1_{\delta+1,k+\eps}$ is non-empty, too. If the saturation of the lattice generated by $D_{p,\delta,k}$ and $v$ is isometric to \begin{footnotesize} $\left( \begin{array}{cc} 2\delta-2+2\eps & b\\ b & 2k-2+4\eps \end{array}\right)$\end{footnotesize}, the saturation of the lattice generated by 
$D_{p,\delta+1,k}$ and $v$ is isometric to \begin{footnotesize} $\left( \begin{array}{cc} 2\delta+2\eps & b-1\\ b-1 & 2k-2+4\eps \end{array}\right)$\end{footnotesize}. Analogously, if $(S,L)$ has genus $p$ and $\{L\}^1_{\delta,k+\eps}$ is non-empty, then $\{L'\}^1_{\delta,k+\eps}$ is non-empty for every primitively polarized $(S',L')$ of genus $p-1$. If the corresponding lattice in the genus $p$ case is \begin{footnotesize} $\left( \begin{array}{cc} 2\delta-2+2\eps & b\\ b & 2k-2+4\eps \end{array}\right)$\end{footnotesize}, the lattice in the genus $(p-1)$ case is \begin{footnotesize} $\left( \begin{array}{cc} 2\delta-2+2\eps & b-1\\ b-1 & 2k-2+4\eps \end{array}\right)$\end{footnotesize}. These remarks along Example \ref{mitico} give us all possible isometry classes of lattices $T$ as in the statement.
\end{proof}

\begin{remark}
As explained in Remark \ref{rem:nonprim}, the above proposition does not give all wall divisors up to the monodromy action. However, when $k-1+2\eps$ is a prime power, we have that $T$ determines and is determined by the monodromy orbit of $D$ as all isometries of $H^2(S^{[k]}_\eps)$ can be extended to isometries of $\Lambda$ fixing $v$. Hence the above proposition gives a full list of wall divisors up to monodromy  in these cases.
\end{remark}

\section{Lazarsfeld-Mukai bundles associated with  nodal curves}\label{sec:LM}

Let $C$ be a nodal curve on an abelian or $K3$ surface $S$ such that its normalization
$\tC$ possesses a {\it complete} $g^1_{k+\eps}$, that is, a globally generated line bundle $A$ of degree $k+\eps$ such that 
$h^0(A)=2$. We denote by $\nu:\tC\to C$ the normalization map, and by $N$ the $0$-dimensional subscheme of the nodes of $C$. By standard facts, $\nu_*A$ is a torsion free sheaf of rank one on $C$ that fails to be locally free precisely at $N$. 

Let $f: \widetilde{S} \to S$ be the blow up of $S$ at $N$, so that we have a commutative diagram
\begin{equation}
  \label{eq:fnf}
\xymatrix{ \tC \ar@{}[r]|-*[@]{\subset} \ar[d]_{\nu} \ar[dr]_{\varphi} & \tS \ar[d]^f \\
             C \ar@{}[r]|-*[@]{\subset} & S}
\end{equation}
We denote by $\E_{\tC,A}$\ and $\E_{{C},\nu_*A}$
the so-called {\it Lazarsfeld-Mukai bundles} associated with the line bundle $A$ on $\tC$ and the torsion free sheaf $\nu_*A$ on $C$, respectively; the duals of these bundles are the kernels of the evaluation maps regarding $A$ and $\nu_*A$ as torsion sheaves on the surfaces, that is, we have the following short exact sequences:
\begin{equation}
    \label{eq:defE}
   \xymatrix{
0 \ar[r] &  \E_{\tC,A}^\vee \ar[r] &  H^0(\tC,A) \* \O_{\tS} \ar[r]^{\hspace{0.9cm}ev_{\tS,A}} & A  \ar[r] & 0,} 
  \end{equation}
and
\begin{equation}
    \label{eq:defE2}
   \xymatrix{
0 \ar[r] &  \E_{C,\nu_*A}^\vee \ar[r] &  H^0(C,\nu_*A) \* \O_S \ar[r]^{\hspace{0.9cm}ev_{S,\nu_*A}} & \nu_*A.} 
  \end{equation}
The right arrow in \eqref{eq:defE2} might be non-surjective, as $\nu_*A$ is not necessarily globally generated (cf. Lemma \ref{lemma:abovelemma}). Pushing forward \eqref{eq:defE} to $\tS$ and using the isomorphisms $H^0(\tC,A) \cong H^0(C,\nu_*A)$ and $f_*\O_{\tS} \cong \O_S$, one shows that
\begin{equation}
  \label{eq:caz}
 \E_{{C},\nu_*A}^{\vee} \cong f_*(\E_{\tC,A}^{\vee}).  
\end{equation}

The following result establishes when \eqref{eq:defE2} is exact on the right.

\begin{lemma} \label{lemma:abovelemma}
Let $C$ be a nodal curve and denote by $\nu:\tC\to C$ the normalization map. Let $A$ be a complete, base point free pencil on $C$. Then the sheaf $\nu_*A$ is globally generated except precisely at the nodes of $C$ that are neutral with respect to $|A|$. 
\end{lemma}

\begin{proof}
We can assume that $C$ has only one node $P$, since the general case is analogous using partial normalizations. Let $\phi_{|A|}:\tC \to \PP^1$ be the 
morphsm defined by $|A|$.

Assume that $P$ is a neutral node with respect to $|A|$. Then $\phi_{|A|}$
factors through a morphism
$\phi: C \to \PP^1$ having the same degree as $\phi_{|A|}$. Having set $A':=\phi^*\O_{\PP^1}(1)$, one has $\nu^*A' \cong
A$ and $\nu_*A \cong \nu_*\nu^*A' \cong A' \* \nu_* \O_C $, hence $\nu_*A$ sits in the following short exact sequence:
$$0\longrightarrow A'\longrightarrow \nu_*A\longrightarrow\O_P\to 0.$$
Since $h^0(C,\nu_*A)= h^0(\tC,A)=h^0(C,A')=2$, the sheaf $\nu_*A$ cannot be globally generated. 

Conversely, assume that $\nu_*A$ is not globally generated at $P$, that is, the evaluation map  $$ev:H^0(\nu_*A) \* \O_C \to \nu_*A$$ is not surjective.
Since $A$ is globally generated and $\nu$ is a finite map, we have a surjection $$H^0(\nu_*A)\otimes \nu_*\O_{\tC} \simeq H^0(A)\otimes \nu_*\O_{\tC}\twoheadrightarrow \nu_*A.$$ Using \eqref{eq:standard}, one can easily show that the cokernel $A_1$ of $ev$ sits in a short exact sequence
$$0\longrightarrow A_1\longrightarrow \nu_*A\longrightarrow\O_P\longrightarrow 0.$$ In particular, $A_1$ is a line bundle and $\nu^*A_1=A$. Hence, the morphism $\phi_{|A|}$ factors through a morphism $\phi_{|A_1|}:C\to\mathbb{P}^1$, which means that $P$ is a neutral node with respect to $|A|$.
\end{proof}

The above lemma implies that if $C$ is a general curve of the nice component $Y_{\delta,k+\eps}$ in Theorem \ref{thm:exist} and  $|A|$ is a general $g^1_{k+\eps}$ on $\tC$, then $\nu_*A$ is globally generated and the short exact sequence \eqref{eq:defE2} is exact on the right. By dualizing it, we obtain:
\begin{equation}
    \label{eq:defEdual}
   \xymatrix{
0 \ar[r] &  H^0(C,\nu_*A)^\vee \* \O_S \ar[r] & \E_{C,\nu_*A} \ar[r] & \ext^1(\nu_*A,\O_S) \ar[r] & 0,}  
  \end{equation}
where $\ext^1(\nu_*A,\O_S)$ is a rank one torsion free sheaf on $C$.
This defines a subspace $$V \cong H^0(C,\nu_*A)^{\vee} \in G(2, H^0(S,\mathcal{E}_{C,\nu_*A}))$$ such that the evaluation map $V\otimes \mathcal{O}_S\to \mathcal{E}_{C,\nu_*A}$ is injective, drops rank along $C$ and has rank $0$ exactly at the nodes of $C$. As a consequence, every section in $V$ vanishes along a $0$-dimensional subscheme of length $k+\eps+\delta$ always containing the subscheme $N$ of the $\delta$ nodes of $C$.

We want to understand whether the pair $(C,\nu_*A)$ univocally determines the subspace $V$; this is equivalent to computing the dimension of  $ \Hom(\E_{C,\nu_*A},\ext^1_{\O_S}(\nu_*A,\O_S))$. To achieve this goal we need some technical results.

For any torsion free rank one sheaf $\A$ on $C$ we denote by $\A^D$ the dual sheaf $\hom_{\O_C}(\A,\O_C)$. We prove the following:

\begin{lemma} \label{lemma:technical2}
 Let $\F$ be any rank one torsion free sheaf on a curve $C\subset S$, where $S$ is a $K3$ or abelian surface. Then one has:
 \begin{equation}
   \label{eq:andreas}
  \omega_C \* \F^D \cong \mathfrak{ext}^1_{\O_S}(\F,\O_S). \end{equation}

Assume furthermore that $\F$ is the cokernel of an injective map
$V\* \O_S \to \E$, where $\E$ is a rank two bundle on $S$ and $V \in G(2,H^0(S,\E))$. If $Z$ is a zero-dimensional scheme that is the vanishing locus of a $s \in V$, one obtains the isomorphisms: 
\begin{equation}
 \label{eq:andreas2} \F \cong \omega_C \* \I_{Z/C} 
\end{equation}
and
\begin{equation}
  \label{eq:andreas3}
\ext^1_{\O_S}(\F,\O_S) \cong \hom_{\O_C}(\I_{Z/C},\O_C).
 \end{equation}
\end{lemma}

\begin{proof}
Applying $\mathfrak{hom}_{\O_S}(\F,-)$ to the short exact sequence
\begin{equation}
  \label{eq:andreas1}
\xymatrix{
0 \ar[r] & \I_{C/S} \ar[r]^{\alpha} & \O_S \ar[r] & \O_C \ar[r] & 0,}
\end{equation}
we obtain
\[
\xymatrix{
0=\mathfrak{hom}_{\O_S}(\F,\O_S) \ar[r] & \mathfrak{hom}_{\O_S}(\F,\O_C) \ar[r] &  
\mathfrak{ext}^1_{\O_S}(\F,\I_{C/S}) \ar[r]^{\alpha'} & \mathfrak{ext}^1_{\O_S}(\F,\O_S).} 
\]
Since $\alpha$ and $\alpha'$ are given by multiplication with the local equation of $C$ and $\F$ is supported precisely at $C$, the map $\alpha'$ is zero. 
Hence, we get
\[\mathfrak{hom}_{\O_C}(\F,\O_C) \cong \mathfrak{hom}_{\O_S}(\F,\O_C) \cong 
\mathfrak{ext}^1_{\O_S}(\F,\I_{C/S}) \cong \mathfrak{ext}^1_{\O_S}(\F,\O_S(-C))
\cong \mathfrak{ext}^1_{\O_S}(\F,\O_S) \* \O_S(-C),\]
and \eqref{eq:andreas} is obtained by tensoring with $\O_S(C)$.

Concerning the second part of the statement, we set  $L :=\det \E$ and consider the following commutative diagram:
\begin{equation} \label{juve30}
\xymatrix{
&  & &  0 \ar[d] & \\
& 0 \ar[d] & &  \O_S \ar[d]^i & \\
0 \ar[r] & \O_S \ar[r]^s \ar[d] & \E \ar[r]  \ar[d]^{\cong} & L \* \I_{Z/S} \ar[r] \ar[d]  & 0 \\
0 \ar[r] & V \* \O_S \ar[r] \ar[d]  & \E \ar[r]   & \F \ar[r] \ar[d]  & 0 \\ 
  & \O_S \ar[d] & & 0 & \\
 & 0 & &  & 
}
\end{equation}
The short exact sequence of ideals
\[
\xymatrix{
0 \ar[r] & \I_{C/S} \ar[r] & \I_{Z/S} \ar[r] & \I_{Z/C} \ar[r] & 0
}
\]
yields the isomorphism $\I_{Z/C} \cong \I_{Z/S} \* \O_C$. Hence, \eqref{eq:andreas2} is obtained by restricting the vertical exact sequence on the right in \eqref{juve30} to $C$ and using that $i$ is given by multiplication with the local equation of $C$. Combining  \eqref{eq:andreas2} and \eqref{eq:andreas}, we obtain
\eqref{eq:andreas3}.
\end{proof}

We now extend \cite[Lemma 2]{P}  to possibly nodal curves on $K3$ or abelian surfaces (cf. also \cite[Pf. of Prop. 3.2]{LC1} concerning other types of irregular surfaces); we are confident that this result of independent interest.

\begin{proposition}\label{prop:pareschi}
Let $C$ be a nodal curve on a $K3$ or abelian surface $S$ and denote by $\nu:\tC\to C$ the normalization map. Let $A$ be a globally generated line bundle on $\tC$ satisfying $h^0(\tC,A)=2$. Then there is a natural short exact sequence 
\begin{equation} \label{eq:brasile}
\xymatrix{
0\ar[r]& \O_{\tC}\ar[r] & \E_{\tC,A}^{\vee} \* \omega_{\tC} \* A^{\vee} \ar[r] & \omega_{\tC} \* (A^{\vee})^{\* 2} \ar[r]&0
}
\end{equation}
whose coboundary map $Q: H^0(\tC, \omega_{\tC} \* (A^{\vee})^{\*2}) \to H^1(\tC,\O_{\tC})$ coincides, up to multiplication by a nonzero scalar factor, with the composition of the Gaussian map
$$\mu_{1,A}:\ker\mu_{0,A}\to H^0(\tC,\omega_{\tC}^{\otimes 2})$$
and the dual of the Kodaira-Spencer map 
\[
\xymatrix{
 \kappa: H^1(C,\nu_*\O_{\tC})^{\vee} \cong T_{[C]}V_{\{L\},\delta} \ar[r] & 
T_{[\tC]}\M_{p_g(C)} \cong H^0(\tC, \omega_{\tC}^{\*2})^{\vee},
}
\]
via the canonical isomorphism $H^1(C,\nu_*\O_{\tC}) \cong H^1(\tC,\O_{\tC})$.
\end{proposition}

\begin{proof}
The exact sequence \eqref{eq:brasile} is obtained as  
\cite[(4)]{P} since $\omega_S\simeq \O_S$.  

We denote by $\N'_{C/S}$ the {\it equisingular normal sheaf} of $C$ in $S$, and by $\N_{\varphi}$ the {\it normal sheaf to the morphism} $\varphi$ in \eqref{eq:fnf}. We recall that $\N'_{C/S}\simeq \nu_*\N_{\varphi} \cong \omega_C \* (\nu_*\O_{\tC})^D$ by, e.g., \cite[Lemma 3.4.15]{ser} and \cite[p. 111]{ta}. In particular, we have:
\begin{equation} \label{eq:forzajuve}
 T_{[C]}V_{\{L\},\delta} \cong H^0(C, \N'_{C/S}) \cong H^0(\tC,\N_{\varphi})  \cong H^1(C,\nu_*\O_{\tC})^{\vee}.
\end{equation}
The Kodaira-Spencer map $\kappa$ is the 
first coboundary map of the {\it normal sequence} 
\begin{equation} \label{eq:normalsequence}   
\xymatrix{
0 \ar[r] &  \T_{\tC} \ar[r] &  \varphi^*\T_S \ar[r] & \N_{\varphi} \cong \omega_{\tC} \ar[r] & 0,}
\end{equation}
where the isomorphism on the right again follows from the triviality of $\omega_S$. 

Applying $\nu_*$ and again using the isomorphism $\nu_*\omega_{\tC} \cong \omega_C \* (\nu_*\O_{\tC})^D$, 
we obtain the exact sequence
\[
\xymatrix{
0 \ar[r] &  \nu_*(\omega_{\tC}^{\vee}) \ar[r] & \T_S \* \nu_*\O_{\tC} \ar[r] & \omega_C \* (\nu_*\O_{\tC})^D \ar[r] & 0,}
\]
and its dual
\[
\xymatrix{
0 \ar[r] &  \omega_C^{\vee} \* (\nu_*\O_{\tC}) \ar[r] &  \Omega_S \* (\nu_*\O_{\tC})^D \ar[r] & 
(\nu_*(\omega_{\tC}^{\vee}))^D \ar[r] & 0.}
\]
The right exactness of the latter is due to the fact that
$\ext^1_{\O_C}((\nu_*\O_{\tC})^D,\O_C)=0$, which can easily be verified using the standard exact sequence 
\begin{equation} \label{eq:standard}
0\longrightarrow \O_C\longrightarrow\nu_*\O_{\tC} \longrightarrow \O_N\longrightarrow 0.
\end{equation}
Tensoring with $\omega_C$, we obtain
\begin{equation} \label{eq:seconda}
\xymatrix{
0 \ar[r] &  \nu_*\O_{\tC} \ar[r] &  \Omega_S \* \omega_C \* (\nu_*\O_{\tC})^D \ar[r] & 
\omega_C \* (\nu_*(\nu_{\tC}^{\vee}))^D \cong \nu_*(\omega_{\tC}^{\* 2}) \ar[r] & 0,}
\end{equation}
where the last isomorphism follows from \cite[Lemma 4.6]{BP}. By construction, the first coboundary map of \eqref{eq:seconda} is $\kappa^{\vee}$. 

We now follow \cite[Pf. of Lemma 1]{P}. Tensoring the derivation operator $d: \O_{\tC} \to \omega_{\tC}$ with the evaluation map $ev_{\tC,A}:H^0(A) \* \O_{\tC} \to A$,
 we obtain a map $H^0(A) \* \O_{\tC} \to \omega_{\tC} \* A$, whose restriction to 
$\ker ev_{\tC,A} \cong A^{\vee}$ is $\O_{\tC}$-linear. Tensoring the restricted map with $\omega_{\tC} \* A^{\vee}$, we obtain a map of $\O_{\tC}$-modules 
\begin{equation} \label{eq:s}
\xymatrix{
s: \; \omega_{\tC} \* (A^{\vee})^{\*2} \ar[r] & \omega_{\tC}^{\*2}, 
}
\end{equation}
whose associated map at the global section level is the Gaussian map $\mu_{1,A}$ (remember the standard isomorphism $\ker\mu_{0,A} \cong H^0(\tC, \omega_{\tC} \* (A^{\vee})^{\*2})$). 

Similarly, tensoring the pullback under $f$ of the derivation operator 
$d: \O_S \to \Omega_S$ with the evaluation map $ev_{\tC,A}: H^0(A) \* \O_{\tS} \to A$,
 we obtain a map $H^0(A) \* \O_{\tS} \to f^*\Omega_S \* A$, whose restriction to $\E_{\tC,A}^{\vee}$ is $\O_{\tS}$-linear. Tensoring the restricted map with $\omega_{\tC} \* A^{\vee}$, we obtain a map of $\O_C$-modules 
\begin{equation} \label{eq:t}
\xymatrix{
t: \;  \E_{\tC,A}^{\vee} \* \omega_{\tC} \* A^{\vee} \ar[r] & f^*\Omega_S \* \omega_{\tC}.
}
\end{equation}
The sequences and maps \eqref{eq:brasile}, \eqref{eq:seconda}, \eqref{eq:s} and \eqref{eq:t} combine into:
\[
\xymatrix{
0\ar[r]& \nu_*\O_{\tC}\ar[r] \ar@{=}[d] &  \nu_* (\E_{\tC,A}^{\vee} \* \omega_{\tC} \* A^{\vee})  \ar[r] \ar[d]^{\nu_*t} & \nu_*(\omega_{\tC} \otimes (A^{\vee})^{\* 2}) \ar[r] \ar[d]^{\nu_*s} & 0 \\
0 \ar[r] &  \nu_*\O_{\tC} \ar[r] &  \Omega_S \* \omega_C \* (\nu_*\O_{\tC})^D \ar[r] & 
\nu_*(\omega_{\tC}^{\* 2}) \ar[r] & 0
}
\]
and the result follows.
\end{proof}

We are now ready to prove the following:

\begin{proposition}\label{tecnica}
If $C$ is a general element of the component $Y_{\delta,k+\eps}$ in Theorem \ref{thm:exist} and  $|A|$ is a general $g^1_{k+\eps}$ on $\tC$, then $$\dim \Hom(\E_{C,\nu_*A},\ext^1_{\O_S}(\nu_*A,\O_S))=1.$$
\end{proposition}

\begin{proof}
 By \eqref{eq:andreas}, we have
$\Hom(\E_{C,\nu_*A},\ext^1(\nu_*A,\O_S)) \cong H^0(\E_{C,\nu_*A}^\vee \* \omega_C \* (\nu_*A)^D)$ and, by \eqref{eq:defEdual} tensored with $\E_{C,\nu_*A}^{\vee}$,  the latter contains $H^0(\E_{C,\nu_*A}^\vee \*\E_{C,\nu_*A})\simeq\mathbb{C}$, where the isomorphism follows from the fact that $\E_{C,\nu_*A}$ is simple when $(S,L) $ is general as in Theorem \ref{thm:exist} (this is standard, cf., e.g.,  \cite{P,ck,klm}, and follows from the fact that $\{L\}$ does not contain reducible or nonreduced members).

If $S$ is $K3$ the result is well-known and due to the fact that $h^1(\O_S)=0$ yields $h^1(\E_{C,\nu_*A}^{\vee})=0$ by \eqref{eq:defE2}. 
It remains to treat the case where $S$ is abelian. We have $\omega_C \* (\nu_*A)^D \cong \nu_*(\omega_{\tC} \* A^{\vee})$ by \cite[Lemma 4.6]{BP}. Hence, by \eqref{eq:caz}, there is a natural morphism
\[ \xymatrix{
\E_{C,\nu_*A}^{\vee} \* \omega_C \* (\nu_*A)^D \cong f_*(\E_{\tC,A}^{\vee}) \* 
   \nu_*(\omega_{\tC} \* A^{\vee}) \ar[r] & \nu_*(\E_{\tC,A}^\vee \* \omega_{\tC} \* A^{\vee}),}
\]
which is injective as the left hand side is torsion free on $C$ and the map is an isomorphism outside of $N$. In particular, we get an inclusion
\[  H^0(C,\E_{C,\nu_*A}^\vee \* \omega_C \* (\nu_*A)^D) \sub H^0(C,\nu_*(\E_{\tC,A}^\vee \* \omega_{\tC} \* A^{\vee})) \cong H^0(\tC,\E_{\tC,A}^\vee \* \omega_{\tC} \* A^{\vee}),\]
and thus it is enough to prove the injectivity of the first coboundary map $Q$ 
of \eqref{eq:brasile}.

If $\rho(p-\delta,1,k+1)\geq 0$, then Theorem \ref{thm:exist} yields $\ker\mu_{0,A}\simeq H^0(\tC, \omega_{\tC} \* (A^{\vee})^{\* 2} )=0$ and $Q$ is automatically injective.

If $\rho(p-\delta,1,k+1)<0$, then $\dim\ker\mu_{0,A}=-\rho(p-\delta,1,k+1)$ because, by Theorem \ref{thm:exist}(iii),  $A$ defines an isolated and reduced point of $G^1_{k+1}(\tC)$. By Proposition \ref{prop:pareschi},
we need to show that $\kappa^\vee\circ \mu_{1,A}$ is injective, or equivalently, $\mu_{1,A}^\vee\circ \kappa$ is surjective. Since $A$ is a pencil, then the map $\mu_{1,A}$ is injective and its dual sits in the following short exact sequence:
\[
\xymatrix{
0 \ar[r] & T_{\left[\tC\right]}\M^1_{g,k+1}  \ar[r] &  T_{\left[\tC\right]}\M_g  \ar[r]^{\hspace{-0.6cm}\mu_{1,A}^\vee} &  \N_{\M^1_{g,k+1}/\M_g}\vert _{[\tC]} \ar[r] &   0,
}
\]
cf. \cite[pp. 807--824]{ACG}. By Theorem \ref{thm:exist}, $Y_{\delta,k+1}$ has codimension $-\rho(p-\delta,1,k+1)$ in the Severi variety $V_{\{L\},\delta}(S)$, hence the image of the natural map $\psi:V_{\{L\},\delta}(S)\to \M_g$ is transversal to $\M^1_{g,k+1}$ around $[C]$. This forces $\mu_{1,A}^\vee\circ \kappa$ to be surjective.
\end{proof}


We now fix notation that will be used in the construction of a component of the locus in $S^{[k]}_\eps$ covered by  rational curves of class $R_{p,\delta,k}$.

Let $C$ and $A$ be as in Proposition \ref{tecnica}. Since $\nu_*A$ is globally generated in this case (by Theorem \ref{thm:exist} and Lemma \ref{lemma:abovelemma}), the {\it Mukai vector} of the Lazarsfeld-Mukai bundle $\mathcal{E}_{C,\nu_*A}$ is 
\begin{equation}\label{mukaivector}
v_{p,\delta,k}:=v(\mathcal{E}_{C,\nu_*A})=(2, c_1(L),\chi+2(\eps-1)),\,\,\,\textrm{with}\,\,\chi:=\chi(\E_{C,\nu_*A})=p-\delta-k+3-5\eps.
\end{equation} 
 
If $\Pic(S)  \cong \ZZ[L]$, which holds off a countable union of proper closed subvarieties of the moduli space of pairs $(S,L)$, 
then $\mathcal{E}_{C,\nu_*A}$ is stable with respect to the polarization $L$ (as in, e.g., \cite[Proposition A.2]{klm}).

Let $\M$ be the component of the moduli space of Gieseker $L$-semistable torsion free sheaves on $S$ with Mukai vector $v_{p,\delta,k}$ as in \eqref{mukaivector} that  contains $\mathcal{E}_{C,\nu_*A}$. We recall that $\M$ is a projective symplectic manifold admitting a symplectic resolution of dimension:
\begin{equation}\label{eq:dimM'}
 \dim \M=2p-4\chi+(1-\eps)8,
\end{equation}
with $\chi$ as in \eqref{mukaivector}.

Every torsion free sheaf $[\E]\in \M$ satisfies $h^2(\E)=0$ because of Serre duality and inequality $\mu_L(\E)>0$. Furthermore, as soon as $h^0(\E)\geq 2$, then for all $V\in G(2,H^0(\E))$ the evaluation map $ev:V\otimes\O_S\to \E$ is injective. Indeed, if this were not the case, its kernel would be isomorphic to $\O_S(-D)$ for an effective divisor $D$ and we would find a short exact sequence: 
$$0\longrightarrow\O_S(D)\longrightarrow\E\longrightarrow \det\E(-D)\otimes I_\xi\longrightarrow0,$$
where $\xi\subset S$ is a $0$-dimensional subscheme. As $\det\E$ is indecomposable and $h^2(\E)=0$, then $D=0$ and this contradicts the fact that $V$ is generated by $2$ linearly independent sections of $\E$.

The following two results determine properties of $\M$ that will play a fundamental role in our construction of uniruled subvarieties of $S^{[k]}_{\eps}$. They can also be seen as applications of Proposition \ref{tecnica} to moduli spaces of stable sheaves on $S$ and are therefore interesting in themselves.

\begin{prop}\label{prop:importante}
 Let $p,\delta,k$ be integers such that \eqref{eq:boundA} is satisfied and let $v$ and $\chi$ be as in \eqref{mukaivector}. Then there is a (nonempty) irreducible component $\M$ of the moduli space of $L$-stable sheaves on $S$ with Mukai vector $v_{p,\delta,k}$ such that the following properties are satisfied:
\begin{itemize}
\item[(i)] If $\chi\geq2\delta+2$, then $h^1(\E)=h^1(\E\otimes I_\tau)=0$ for a general pair $([\E],\tau)\in\M \x S^{[\delta]}$. 
\item[(ii)] If $\chi<2\delta+2$, the locus 
$$X:=\{([\E],\tau)\in \M  \x  S^{[\delta]}\,\,\vert\,\,h^0(\E\otimes I_\tau)\geq 2\}$$ is nonempty, with an irreducible component $X_0$ whose general point satisfies $h^0(\E\otimes I_\tau)=2$ and $h^1(\E\otimes I_\tau)=2+2\delta-\chi$. Furthermore, $X_0$ is birational to a component of the relative Brill-Noether variety 
$$
\G^1_{k+\eps}\left(V^{k+\eps}_{\{L\},\delta}\right):=\{([C],\mathfrak{g})\,|\,[C]\in V^{k+\eps}_{\{L\},\delta},\, \mathfrak{g}\in G^1_{k+\eps}(\tC),\textrm{ with }\tC\textrm{ the normalization of }C\},
$$
having the expected dimension $2(k-1+\eps)$. \color{black}
\end{itemize}
\end{prop}

\begin{proof}
Let $\M$ contain $\E_{C,\nu_*A}$ with $C$ and $A$ as in Proposition \ref{tecnica} and let $X\subset \M  \x  S^{[\delta]}$ parametrize pairs $([\E],\tau)$ such that $h^0(\E\otimes I_\tau)\geq 2$. Trivially, $X$ coincides with $\M \x  S^{[\delta]}$ as soon as $\chi\geq2\delta+2$. If instead $\chi<2\delta+2$, then $X$ is a closed subscheme which can be defined using fitting ideals and its expected codimension is $2(2\delta-\chi+2)$, whence its expected dimension is $2(k-1+\eps)$.
 Furthermore, $X$ is nonempty because $([\E_{C,\nu_*A}],N)$ lies in it. 
 
 Let $\G$ be the family of triples $([\E],\tau,V)$ with $([\E],\tau)\in X$ and $V\in G(2,H^0(\E\otimes I_\tau))$, and let $p:\G\to X \x  S^{[\delta]}$ be the natural projection. Existence of $\G$ follows from existence of a moduli space of $\O_S$-stable coherent systems; indeed, since $\M$ parametrizes torsion free sheaves of rank $2$ and indecomposable first Chern class, then any pair $([E],V)$ with $[E]\in\M$ and $V\in G(2, H^0(E))$ is automatically an $\O_S$-stable coherent system (proceed as in \cite[\S 3]{LC2}). 

Since $([\E_{C,\nu_*A}],N,H^0(C,\nu_*A)^\vee)$ lies in $\G$, then for a general element $([\E],\tau,V)\in \G$ the evaluation map $ev:V\otimes \O_S\to\E$ is injective and drops rank along a $\delta$-nodal curve $ \Gamma$, which is singular precisely at the support of $\tau$. Furthermore, the cokernel $B$ of $ev$ is torsion free of rank $1$ on $\Gamma $ and is not locally free exactly along $\mathrm{Sing}(\Gamma)$. We set $B_1:=\ext^1(B,\O_S)$. If  $\eta :\widetilde{ \Gamma}\to  \Gamma$ is the normalization map, we can write $B_1={ \eta}_*A_1$ for some $A_1\in\Pic^{k+ \eps \color{black}}(\widetilde{ \Gamma })$. From the short exact sequence
\begin{equation} \label{eq:conv}
0\longrightarrow V\otimes\O_S\longrightarrow\E\longrightarrow B\longrightarrow 0,
\end{equation}
we get that $H^0(A_1)^\vee\simeq H^0(B_1)^\vee\simeq H^1(B)\twoheadrightarrow V$ and hence $\mathfrak{g}=(A_1, V^\vee)$ defines a $g^1_{k+ \eps }$ on $\widetilde{\Gamma}$. We thus obtain a rational map $h:\G\dashrightarrow \G^1_{k+ \eps }(V^{k+ \eps }_{\{L\},\delta})$. Our hypotheses, along with Theorem \ref{thm:exist}, ensure that $(C,\nu_*A)$ lies in a component $Z$ of the image of $h$ of dimension: 
\begin{equation}\label{dimensione}
\dim\,Z=\dim\,V^{k+ \eps}_{\{L\},\delta}+\max\{0,\rho(p-\delta,1,k+ \eps)\}=2(k-1+\eps).
\end{equation}
On the other hand, Proposition \ref{tecnica} implies that $h$ is generically injective and hence, for a general $([\E],\tau)\in X$, we have:
\begin{equation}\label{parametri}
\dim\,Z=\dim\, X+2(h^0(\E\otimes I_\tau)-2).
\end{equation}

If $\chi\geq 2\delta+2$, then $X= \M \x S^{[\delta]}$ and \eqref{eq:dimM'}, \eqref{dimensione} and  \eqref{parametri}yield:
\begin{equation}
2(k-1+\eps)  =  \dim  Z  \geq  \dim\M+\dim S^{[\delta]}+2(\chi(\E\otimes I_\tau)-2)=2(k-1+\eps);
\end{equation}
thus, equality holds and a general $([\E],\tau)\in \M \x  S^{[\delta]}$ satisfies $h^1(\E)=h^1(\E\otimes I_\tau)=0$.

If instead $\chi<2\delta+2$, then $\mathrm{expdim}\,X= 2(k-1+\eps)$ and, by \eqref{dimensione} and \eqref{parametri}, one obtains  $\dim X=\mathrm{expdim}X$ and $h^0(\E\otimes I_\tau)=2$ for a general $([\E],\tau)\in X$.
\end{proof}
\begin{lemma}\label{h1}
Under the same hypotheses as in Proposition \ref{prop:importante}, assume moreover that $S$ is abelian and $\chi\geq 4$. Then a general $[\E]\in\M$ satisfies $h^1(\E\* \L_0)=0$ for all $\L_0\in\Pic^0(S)$.
\end{lemma}
\begin{proof}
It is enough to show that the locus $\F:=\{[\E]\in\M\,|\, h^1(\E)\neq 0\}$ has codimension greater than two in $\M$. We perform a parameter count as in \cite[Prop. 4.2 and \S 7]{LC0}. Let $\G_1$ be the parameter space of extensions
\begin{equation}\label{extension}
0\longrightarrow\O_S\longrightarrow\E'\stackrel{\beta}{\longrightarrow}\E\longrightarrow0,
\end{equation}
with $[\E]\in \F$. Since $\Hom(\E,\O_S)=0$ for all $\E\in\M$, the fibre of the natural map $\pi_1:\G_1\to\F$ over $[\E]$ is isomorphic to $\mathbb{P}(H^1(\E))$. A sheaf $\E'$ as in \eqref{extension} is stable; let $v':=v_{p,\delta,k}+(1,0,0)$ be its Mukai vector and denote by $\pi_2:\G_1\to \M(v')$ the natural projection mapping \eqref{extension} to $[\E']$. The fibre of $\pi_2$ over $[\E']$ is the Quot-scheme $\mathrm{Quot}_S(\E',P)$, where $P$ is the Hilbert polynomial of $\E$. The following upper bound for the dimension of $\mathrm{Quot}_S(\E',P)$ at $[\beta:\E'\to \E]$ is well-known:
$$
\dim_{[\beta]}\mathrm{Quot}_S(\E',P)\leq \dim\Hom(\O_S,\E)=h^0(\E).$$
It follows that:
$$
\dim\M-2\chi=\dim \M(v')\geq\dim\mathrm{Im}\pi_2\geq \dim\F-\chi-1,
$$
and $\codim_{\M}\F\geq\chi-1>2$ as soon as $\chi\geq 4$.
\end{proof}

The next lemma concerns every component of any moduli space of rank-$2$ Gieseker semistable torsion free sheaves on $S$.
\begin{lemma}\label{badlocus}
Let $S$ be an abelian or $K3$ surface and let $\M$ be a component of the moduli space of rank-$2$ Gieseker semistable torsion free sheaves on $S$ with Mukai vector \mbox{$v=(2,c_1, \chi+2(\eps-1))$}. For every $[\E]\in\M$, denote by $\mathfrak{S}(\E)$ the cokernel of the injection $\E\hookrightarrow \E^{\vee\vee}$ and by $l_{\E}$ its length. Then every irreducible component of the locus
\begin{equation}\label{nlf}
\M_q:=\{[\E]\in\M\,|\,l_{\E}=q\}
\end{equation}
has codimension at least $q$ in $\M$.
\end{lemma}
\begin{proof}
Assume that $\M_q$ is non-empty. One defines a map $\alpha:\M_q\to\M(v_q)$, where $v_q:=v+(0,0,q)$, which maps $[\E]\in \M$ to $[\E^{\vee\vee}]\in \M(v_q)$. The fibre of $\alpha$ over a general vector bundle $F\in \mathrm{Im}\alpha$ is isomorphic to the Quot-scheme $\mathrm{Quot}_S(F,q)$ of zero dimensional quotient sheaves of $F$ of length $q$; by \cite[Prop. 6.0.1]{kieran}, this Quot-scheme has dimension at most $3q$. Hence,
\begin{equation}
\dim \M-4q=\dim\M(v_q)\geq\dim\mathrm{Im}\alpha=\dim\M_q-\mathrm{Quot}_S(F,q)\geq\dim\M_q-3q,
\end{equation} 
and this concludes the proof.
\end{proof}

\section{Algebraically coisotropic subvarieties of IHS manifolds}\label{sec:sei}

We are now ready to prove our results concerning existence of uniruled subvarieties of $S^{[k]}_{\eps}$. Let $(S,L)$ be a {\it very general} primitively polarized $K3$ or abelian surface of genus $p$, in the sense that it satisfies Theorem \ref{thm:exist} and $\Pic(S) \cong \ZZ[L]$. For $p,\delta,k$ satisfying \eqref{eq:boundA}, let $[C] \in Y_{\delta,k+\eps}$ be general and let $A$ be a general complete $g^1_{k+\eps}$ on its normalization.

Let $\M$ be the component of $\M(v_{p,\delta,k})$ as in the previous section, with $[\E_{C,\nu_*A}]\in \M$. Assume $\chi \geq 2\delta+2$ and set $X:= \M \x S^{[\delta]}$. We denote by $\P$ the parameter space for triples $([\E],\tau,[s])$ with $([\E],\tau)\in X$ and $[s]\in\mathbb{P}(H^0(\E\otimes I_\tau))$; existence of $\P$ again follows from existence of a moduli space of $\O_S$-stable coherent systems and there is a natural map $q:\P\to X$. For any $([\E],\tau,[s]) \in \P$, the section $s$ vanishes along a finite set because otherwise we would have $\mathrm{Hom}(\det\E,\E^{\vee\vee})\neq0$ and this would contradict the $\mu_L$-stability of $\E$. Hence, we have a short exact sequence:
\begin{equation}\label{ses2'}
\xymatrix{
0 \ar[r] & \O_S \ar[r]^s & \E \ar[r] & \det \E\otimes I_{W_s} \ar[r] & 0,
}
\end{equation}
where $W_s$ is a $0$-dimensional subscheme of $S$ of length $k+\epsilon+\delta$ containing $\tau$. This provides a short exact sequence
\begin{equation}\label{ses3'}
\xymatrix{
0 \ar[r] & \eta_s\ar[r] & \O_{W_s} \ar[r] & \O_{\tau} \ar[r] & 0,
}
\end{equation}
defining $\eta_s$, which is a torsion sheaf whose support is contained in that of ${W_s}$. 

If $\E$ is locally free, the scheme  ${W_s}$ is a 
local complete intersection as it is the zero scheme of $s$. It follows that $\ext^2_{\O_S}(\O_{W_s},\O_S)\cong
\O_{W_s}$ by, e.g., \cite[p. 36]{Fr}.
Applying the functor $\hom_{\O_S}(-,\O_S)$ to \eqref{ses3'}
we therefore obtain a surjection 
$$
\xymatrix{
\O_{W_s} \ar[r] & \ext^2_{\O_S}(\eta_s,\O_S),
}
$$
which yields the existence of a subscheme $Z_s\subset {W_s}$ of length $k+ \eps$ such that $\ext^2_{\O_S}(\eta_s,\O_S) \cong \O_{Z_s}$. This defines a rational map 
\begin{equation}\label{map}
g': \P\dashrightarrow S^{[k+ \eps]}
\end{equation}
mapping a point $([\E],\tau,[s])$, with $\E$ locally free, to $Z_s$. We want to extend $g'$ in codimension $1$. 

Pick a sheaf $[\E]\in \M$ that is not locally free and such that $\mathfrak{S}(\E)$ has length one, i.e., $\mathfrak{S}(\E)\simeq \O_P$ for some $P\in S$. A section $s$ of $\E$ gives a section $s'$ of $\E^{\vee\vee}$ and, having denoted by $W_{s'}$ the vanishing locus of $s'$, one has a short exact sequence:
$$
0\longrightarrow\O_P\longrightarrow\O_{W_s}\longrightarrow\O_{W_{s'}}\longrightarrow0.
$$
If both $[\E]\in\M_1$ (cf. \eqref{nlf}) and $s\in H^0(\E)$ are general, then $W_{s'}$ does not contain $P$ and $P$ is a general point on $S$. Indeed, $H^0(\E)$ is not contained in $H^0(\E^{\vee\vee}\*I_P)$ as soon as $H^1(\E^{\vee\vee}\*I_P)=0$; the vanishing follows from the fact that a general $[F]\in\M(v(\E^{\vee\vee}))$ is generically generated by global sections. For $s\in H^0(\E\*I_\tau)$ with $\tau\in S^{[\delta]}$ such that $P\not\in\Supp(\tau)$, we get $\O_{W_s} = \O_{W_{s'}} \+ \O_P$, whence $W_s$ is a local complete intersection, and we  may find a subscheme  $Z_s\subset {W_s}$ of length $k+ \eps$ as before and set $g'([\E],\tau,[s])=Z_s$ also in this case.

We set 
$$
\M^\circ:=\{[\E]\in \M\setminus\cup_{q\geq 2}\M_q\,|\, h^1(\E^{\vee\vee}\* I_{P})=0\textrm{ if } \mathfrak{S}(\E)\simeq\O_P\},
$$
with $\M_q$ as in \eqref{nlf}; then $\M^\circ$ is open in $\M$ and its complement has codimension at least two by Lemma \ref{badlocus}.  We define 
$$
X^\circ:=\{([\E],\tau)\in X\,|\,[\E]\in \M^\circ, \Supp(\tau)\cap \Supp(\mathfrak{S}(\E))=\emptyset\};
$$ 
the complement of $X^\circ$ in $X$ has codimension at least two. Let $\P^{\circ}$ be  the open subscheme of $q^{-1}(X^\circ)\subset \P$ consisting of those triples $([\E],\tau,s)$ such that either $\E$ is locally free, or the vanishing locus $W_{s'}$ as above does not contain the singular locus of $\E$.  Then, the complement of  $\P^{\circ}$ in $\P$ has codimension at least two and, by the above discussion, the rational map $g'$ in \eqref{map} is well-defined on the whole $\P^\circ$. We denote by 
\begin{equation}\label{map1}
\xymatrix{g: \P^\circ  \ar[r] & S^{[k+ \eps]}}
\end{equation}
the restriction of $g'$ to $\P^\circ$.

We prove the following result.
\begin{thm}\label{thm:contrazioni}
Let $(S,L)$ be a very general primitively polarized $K3$ or abelian surface  of genus $p \geq 2$.
Let $0\leq \delta \leq p-2\eps$ and $k \geq 2$ be integers satisfying
\begin{equation} \label{eq:nuovobound}
 \max\{2\delta+2,4\eps\} \leq \chi:=p-\delta-k+3-5\eps \leq \delta+k+1.
\end{equation} 

Then the morphism $g$ is generically injective. In particular, the locus $\mathrm{Locus}(\mathbb{R}^+R_{p,\delta,k}) \subset S_{\eps}^{[k]}$, with $R_{p,\delta,k}$ as in \eqref{class}, has an irreducible component that is birational to a $\mathbb{P}^{\chi-2\delta-1}$-bundle on a holomorphic symplectic manifold of dimension $2(k+1+2\delta-\chi)$.
\end{thm}

\begin{proof}
 First of all, note that when $\chi \geq 2\delta+2$, condition \eqref{eq:boundA} is equivalent to $\chi\leq\delta+k+1$. Indeed, the latter inequality is a rewrite of \eqref{eq:boundB} with $l=1$; on the other hand, if $2\delta+2\leq\chi\leq\delta+k+\eps$, then $\delta\leq k+\eps-2$ and $\alpha=1$ in \eqref{eq:boundA}. Hence, \eqref{eq:nuovobound} ensures that the hypotheses of Proposition \ref{prop:importante} are satisfied and a general fibre of $q:\P^\circ\to X^\circ$ is isomorphic to $\mathbb{P}^{\chi-2\delta-1}$. We denote by $T$ the closure of the image of the morphism $g$ in \eqref{map1}. The proof proceeds by steps.

\vspace{0.3cm}

\noindent{\em STEP I: The morphism $g$ is injective when restricted to a general fibre of the projection $q$.}

Let $Z=g(([\E],\tau,[s]))\in  T$ with $([\E],\tau)\in X^\circ$ general, and denote by $W_s$ the zero scheme of $s$. The fibre of $g|_{q^{-1}([\E],\tau)}$ over $Z$ is contained in $\mathbb{P}(\mathrm{Hom}(\E,\det\E\otimes I_{W_s}))$. This projective space is a point because $\mathrm{Hom}(\E,\det\E\otimes I_{W_s})\simeq H^0(\E\otimes \E^\vee) \cong \CC$ by stability, exact sequence \eqref{ses2'} and Proposition \ref{prop:importante}(i), which yields $h^1(\E)=0$.

\vspace{0.3cm}

\noindent{\em STEP II: The map $g$ is generically injective when restricted to a general fibre of the projection $p_1:\P^\circ\to S^{[\delta]}$. If $\delta=0$, then $g$ is injective.}

Let $Z=g(([\E_1],\tau,[s_1]))\in  T$ with $([\E_1],\tau,[s_1])\in \P^\circ$ general; in particular, $\Supp(\tau)$ is disjoint from $\Supp(Z)$ and $\E_1$ satisfies $h^1(\E_1\otimes\L_0)=0$ for all $\L_0\in\Pic^0(S)$ by Lemma \ref{h1}. By contradiction, assume the existence of $([\E_2],\tau,[s_2])\in\P^\circ$ with $\E_2\not\simeq\E_1$ such that $g(([\E_2],\tau,[s_2]))=Z$. In particular, one has $W_{s_1}=W_{s_2}$. We remark that $\det\E_1\not\simeq\det\E_2$ because otherwise we would have $h^1(\det\E_i\otimes I_{W_{s_i}})>1$ and, by considering the long exact sequence in cohomology associated with \eqref{ses2'} for $\E_1$, we would get a contradiction with $h^1(\E_1)=0$. We tensor \eqref{ses2'} for $\E_1$ with $(\det\E_1)^\vee\otimes \det\E_2$, thus obtaining:
$$
0\to (\det\E_1)^\vee\otimes \det\E_2\longrightarrow \E_1\otimes (\det\E_1)^\vee\otimes \det\E_2\longrightarrow \det\E_2\otimes I_{W_s}\longrightarrow 0.
$$
This yields the contradiction $H^1(\E_1\otimes (\det\E_1)^\vee\otimes \det\E_2)\simeq H^1(\det\E_2\otimes I_{W_s})\neq 0$.

\vspace{0.3cm} 
\noindent{\em STEP III:  The map $g$ is generically injective when restricted to a general fibre of the projection $p_2:\P^\circ\to \M^\circ\subset \M$.}

Let $Z=g(([\E],\tau_1,[s_1]))\in  T$ for a general $([\E],\tau_1,[s_1])$. By contradiction, assume the existence of a subscheme $\tau_2\in S^{[\delta]}$ different from $\tau_1$ 
and a section $s_2\in H^0(\E\*I_{\tau_2})$ such that $g(([\E],\tau_2,[s_2]))=Z$. The evaluation map $ev:\langle s_1,s_2\rangle\otimes\O_S\to \E$ is injective and drops rank along an integral curve $\Gamma\in\{L\}$ of geometric genus $\leq p-k-\eps$ (indeed, $\Gamma$ is singular along $Z$). If $B$ is the cokernel of $ev$, then $B_1:=\ext^1(B,\O_S)=n_*A_1$, where $n:\widetilde{\Gamma}\to \Gamma$ is a partial normalization of $\Gamma$ with $p_a(\widetilde{\Gamma})=p-k-\eps-h$ for some $h\geq 0$ and $A_1$ a complete $g^1_{\delta-h}$ on $\widetilde{\Gamma}$. As a consequence, there is a subscheme $\tau_1'\subset\tau_1$ of length $\delta-h$ and a rational curve $R'$ in $S^{[\delta-h]}$ passing through $[\tau_1']$. Starting from $R'$, one easily constructs a rational curves in $S^{[\delta]}$ passing through $\tau_1$ in contradiction with the generality of $\tau_1$.
\vspace{0.3cm} 

\noindent{\em STEP IV: The map $g$ is generically finite.}

This follows from the previous steps and Proposition \ref{bundle_finito}, which can be applied because the complement of $\P^\circ$ in $\P$ has codimension at least two and $X$ is holomorphic symplectic.
\vspace{0.3cm} 

\noindent{\em STEP V: Let $\P_i=\mathbb{P}(H^0(E_i\*I_{\tau_i}))$ for $i=1,2$ with $([E_1],\tau_1),([E_2],\tau_2)\in X^\circ$ distinct points such that $h^1(\E_i)=0$ for $i=1,2$. Then $g$ does not identify $\P_1$ and $\P_2$.}

By contradiction, assume that $g(\P_1)=g(\P_2)$. As $g|_{\P_i}$ is injective for $i=1,2$ by Step I, it is easy to verify that, if $\ell_1\subset \P_1$ is a general line, then $\ell_2:=g^{-1} (g(\ell_1))\cap \P_2$ is a line, too. By generality, $\ell_1=\mathbb{P}V_1$ corresponds to a pair $(C_1,\nu_*A_1)$, where $C_1$ is a $\delta$-nodal curve and $A_1$ is a $g^1_{k+\eps}$ on its normalization. Having set $\ell_2=\mathbb{P}V_2$, the evaluation map $V_2\otimes \O_S\to\E_2$ drops rank along a curve $C_2$ singular along $\tau_2$. Then, $C_1=C_2$ as both coincide with the image in $S$ of the incidence variety
$$
I:=\{(Z,P)\in g(\ell_1)\times S\subset S^{k+\eps}\times S\,|\, P\in \Supp(Z)\}.
$$
As a consequence, $\tau_2=\tau_1$ and $\E_2\simeq\E_{C,\nu_*A_1}\simeq \E_1$.

\vspace{0.3cm} 

\noindent{\em STEP VI: The morphism $g$ is generically injective.}

By contradiction, assume that  for a general $Z\in T$ there exist at least two distinct points $([\E_1],\tau_1), ([\E_2],\tau_2)\in X^\circ$ such that $Z\in g(\P_1)\cap g(\P_2)$, where $\P_i=\mathbb{P}(H^0(\E_i\*I_{\tau_i}))$; we may assume that $h^1(\E_1)=h^1(\E_2)=0$. The previous step then implies that $g(\P_1) \neq g(\P_2)$. We denote by $\pi:\widetilde{T}\dashrightarrow B$ the maximal rational quotient of the desingularization $\widetilde{T}$ of $T$, and by $\widetilde{Z}$ the inverse image of $Z$ in $\widetilde{T}$. Since a general fibre of $\pi$ is irreducible and $\pi^{-1}(\pi(\widetilde{Z}))$ contains the strict transforms of both $g(\P_1)$ and $g(\P_2)$, then $\dim \pi^{-1}(\pi(\widetilde{Z}))\geq\chi-2\delta>\codim_{S^{[k+\eps]}} T$; this contradicts \cite[Thm. 4.4]{AV} (when $\eps=1$, in order to apply the mentioned result, one needs to pass to the fibers of the Albanese map and use that $g(\P_1) \cup g(\P_2)$ is contained in such a fiber).
\vspace{0.3cm} 

When $\eps=0$, Step VI concludes the proof. If $\eps=1$, consider the composition of $g$ with the Albanese map $\Sigma_k:S^{[k+1]}\to S$. Since $\Sigma_k\circ g$ is constant when restricted to any fibre of $q$, it induces a morphism $F:X^\circ\to S$ that factors, by the universal property of the Albanese variety, through a map $f:\mathrm{Alb}(X)\to S$. One can easily show that both $F$ and $f$ are surjective. The inverse image $g^{-1}(S^{[k]}_\eps)\subset \P^\circ$ is generically a $\mathbb{P}^{\chi-2\delta-1}$-bundle on the inverse image  $(\mathrm{alb}_X|_{X^\circ})^{-1}(\ker f)$, where $\mathrm{alb}_X$ is the Albanese map of $X$. Since $\mathrm{Alb}(X)$ is the product of copies of $S$ and $S^\vee$, the same holds for any connected component of $\ker f$. Therefore, any component of $(\mathrm{alb}_X|_{X^\circ})^{-1}(\ker f)$ is holomorphic symplectic of dimension equal to $\dim X-2=2(k+1+2\delta-\chi)$, and the statement follows. 
\end{proof}

We now give a first application to Conjecture \ref{voisin}.

\begin{cor}\label{cor1}
With the same hypotheses and notation as in Theorem \ref{thm:contrazioni}, set $r:=\chi-2\delta-1$. Then, $\mathbb{S}_r(S^{[k]}_\eps)$ has a ($2k-r$)-dimensional component which is an algebraic coisotropic subvariety of $S^{[k]}_\eps$ covered by curves of class $R_{p,\delta,k}$. 
\end{cor}

\begin{proof}
This follows directly from Theorem \ref{thm:contrazioni} and \cite[Thm. 1.3]{voi}, which states that a closed subvariety of an IHS manifold $X$ contained in $\mathbb{S}_r(X)$ has codimension at least $r$.
\end{proof}

Starting from the closure of $\im g\subset S^{[k+\eps]}$, with $g$ as in \ref{map1}, and then applying the natural rational map $S^{[k+\eps]} \x S^{[l-k]} \dasharrow S^{[l+\eps]}$, one obtains subvarieties of $S_{\eps}^{[l]}$ for any $l\geq k$. We use this observation in order to construct subvarieties of $S_{\eps}^{[k]}$, with $k$ fixed, of codimension $r$ for several values of $r$:

\begin{thm}\label{thm:contrazioni2}
Let $(S,L)$ be a very general primitively polarized $K3$ or abelian surface of genus $p \geq 2$ and fix an integer $k \geq 2$. 

Then, for any integer $r$ satisfying
\begin{equation}
  \label{eq:terminator}
  1 \leq r \leq \min\left\{2k-5-\frac{p-5\eps}{2}, \frac{p-5\eps}{2}+1\right\}, 
\end{equation}
and
\begin{equation}
  \label{eq:terminator2}
  \mbox{$p \geq 9$ if $(\eps,r)=(1,1)$}; \; \; \mbox{$p \geq 11$ if $(\eps,r)=(1,2)$,}  
\end{equation}
there is an algebraically coisotropic subvariety of codimension $r$ in $S^{[k]}_{\eps}$ that is a component of $\mathbb{S}_r(S^{[k]}_\eps)$ and is birational to a $\PP^r$-bundle over a holomorphic symplectic manifold. 

More precisely, for any integer $\delta$ satisfying
\begin{equation}
  \label{eq:condelta}
  \max\left\{0,\frac{p-5\eps+2-r-k}{3}\right\} \leq \delta \leq \frac{p-5\eps+2-2r}{4}, \; \; \mbox{and $\delta >0$ if $r \leq 2$ and $\eps=1$},
\end{equation}
there is such a subvariety $\Sigma_{r,\delta}$ whose lines have class $L-[2(p-2\delta-2\eps)-r+1]\mathfrak{r}_k$. 
\end{thm}

\begin{proof}
We note that if $\delta$ and $k$ are integers satisfying the conditions in Theorem \ref{thm:contrazioni}, then, for any integer $l \geq k$, the image of $\im g \x S^{[l-k]} \subset  S^{[k+\eps]} \x S^{[l-k]} $ under the natural birational map
$S^{[k+\eps]} \x S^{[l-k]} \dasharrow S^{[l+\eps]}$ has the same codimension $r$ as $\im g \subset  S^{[k+\eps]}$. The intersection of this image with any fibre of the Albanese map is birational to a $\PP^{r}$-bundle over a holomorphic symplectic manifold and the coefficients in the class of the lines with respect to the canonical decomposition \eqref{eq:N1} remain unchanged. Hence, Theorem \ref{thm:contrazioni} yields that for any pair of integers $(\delta,k')$ such that 
\begin{equation}
  \label{eq:v1}
2 \leq k' \leq k
\end{equation}
and
\begin{equation}
  \label{eq:condi}
\max\{2\delta+2,4\eps\} \leq p-5\eps -\delta-k'+3\leq \delta+k'+1, \; \; \delta \geq 0
\end{equation}
there is a subscheme of codimension $r:=p-5\eps-3\delta-k'+2$ satisfying the desired conditions. Rewriting \eqref{eq:v1} and \eqref{eq:condi} in terms of $r$ instead of $k'$ yields, respectively,
\begin{equation}
  \label{eq:condelta'}
\frac{p-5\eps+2-r-k}{3} \leq \delta \leq \frac{p-5\eps+r}{3}
\end{equation}
and
\begin{equation}
  \label{eq:condi'}
  \frac{1}{2} \max\{0,4\eps-r-1\} \leq \delta \leq \frac{p-5\eps+2-2r}{4},
\end{equation}
One easily verifies that these two conditions are equivalent to \eqref{eq:condelta}. 
It only remains to prove that for each $r$ satisfying \eqref{eq:terminator} and \eqref{eq:terminator2}, the conditions \eqref{eq:condelta} are nonempty. 

The inequality 
$r \leq 2k-5-\frac{p-5\eps}{2}$ ensures that 
\[ \frac{p-5\eps+2-r-k}{3} \leq \frac{p-5\eps+2-2r}{4} -\frac{3}{4},\]
 thus guaranteeing that the interval $\Big[\frac{p-5\eps+2-r-k}{3}, \frac{p-5\eps+2-2r}{4}\Big]$ contains an integer. 
The condition $r \leq\frac{p-5\eps}{2}+1$ is equivalent to $p-5\eps+2-2r \geq 0$ and thus guarantees that the above interval contains a nonnegative integer.
 We are therefore done, except for the cases $\eps=1$ and $r \leq 2$, where the requirement is that $\delta >0$ by \eqref{eq:condelta}, that is, that
$p-5\eps+2-2r \geq 4$, which is precisely \eqref{eq:terminator2}. 
\end{proof}

We now perform a different construction in order to exhibit uniruled subvarieties of $S_{\eps}^{[k]}$ of any allowed codimension, except codimension $k$ for $\eps=1$. This provides very strong evidence for Conjecture \ref{voisin}. 

\begin{thm} \label{thm:nuovofigo}
    Let $(S,L)$ be a general primitively polarized $K3$ or abelian surface of genus $p \geq 2$ and fix an integer $k \geq 2$. Then for any integer $r$ such that $1 \leq r \leq k-\eps$ there is  an algebraically coisotropic subvariety of codimension $r$ in $S^{[k]}_{\eps}$ that is a component of $\mathbb{S}_r(S^{[k]}_\eps)$ and is birational to a $\PP^r$-bundle; furthermore, the maximal rational quotient of it desingularization has dimension  $2(k-r)$.

More precisely, for any integer $k'$ such that $r+\eps \leq k' \leq \min\{k,p+r-\eps\}$,  
there is such a subvariety $W_{r,k'}$ whose lines have class $L-[2(k'+\eps)-r-1]\mathfrak{r}_k$. 
\end{thm}

\begin{proof}
  For any $r$ and $k$ as in the statement, set $g:=k'-r+\eps$ and $\delta:=p-g$. Note that $2\eps \leq g \leq p$. Since $\rho(g,1,k+\eps)\geq 0$, the Brill-Noether locus  $\{L\}^1_{\delta,k+\eps}$ coincides with the $g$-dimensional Severi variety $V_{\{L\},\delta}$ of genus $g$ nodal curves. For any component $V$ of 
$V_{\{L\},\delta}$, we denote by $\mathcal{C} \to V$ the universal family and by $\widetilde{\mathcal{C}} \to V$ the simultaneous desingularization of all curves in $\mathcal{C}$ (the latter exists, as $V$ is smooth). Let $\Sym^{k'+\eps}(\widetilde{\mathcal{C}}) \to V$ be the relative $(k'+\eps)$-symmetric product, with fibre over a point $[C] \in V$ equal to $\Sym^{k'+\eps}(\widetilde{C})$, where $\widetilde{C}$ is the normalization of $C$. 
By surjectivity of the Abel map $\Sym^{k'+\eps}(\widetilde{C})\to \Pic^{k'+\eps}(\widetilde{C})$, the line bundle $\O_{\widetilde{C}}(x_1+\cdots+ x_{k'+\eps})$  is nonspecial if $x_1, \ldots, x_{k'+\eps}\in\widetilde{C}$ are general . Therefore, 
$\dim |\O_{\widetilde{C}} (x_1+\cdots+ x_{k'+\eps})| = k'+\eps-g =r\geq 1$ and $\Sym^{k'+\eps}(\widetilde{C})$ is  generically a $\PP^{r}$-bundle over a dense, open subset of $\Pic^{k'+\eps}(\widetilde{C})$, which has dimension $g$. It follows that $\Sym^{k'+\eps}(\widetilde{\mathcal{C}})$ is generically a $\PP^{r}$-bundle over a scheme of dimension $g+\dim V=2g$. 

Consider the natural composed morphism 
\[
f: \Sym^{k'+\eps}(\widetilde{\mathcal{C}}) \to \Sym^{k'+\eps}(\mathcal{C}) \to  \Sym^{k'+\eps}(S).
\]
We first prove that $f$ is generically injective. Clearly, for each curve $[C] \in V$ with normalization $\widetilde{C}$, the restriction of $f$ to $\Sym^{k'+\eps}(\widetilde{C})$ is generically injective. Hence, it is enough to show that, if $[Z] \in \mathrm{Im}f$ is general, then the scheme $Z$ is contained in a unique curve parametrized by $V$. A general $[Z] \in \mathrm{Im}f$ consists of 
$k'+\eps$ general points on a general curve $C_0$ parametrized by $V$. For a general point $x_1 \in C_0$, the set $\{ [C] \in  V \; | \; x_1 \in C\}$ has codimension one in $V$, as $C_0$ is not a common component of all curves parametrized by $V$. Proceeding inductively, assume that for a fixed $1 \leq j \leq g-1$ we have chosen $j$ distinct points $x_1, \ldots, x_j \in C_0$ such that the set $\{ [C] \in  V \; | \; x_1\ldots, x_j  \in C\}$ has codimension $j$ in $V$. Again, as $C_0$ is not a component of any curve in this set different from $C_0$, for a general $x_{j+1} \in C_0$
the set $\{ [C] \in  V \; | \; x_1\ldots, x_{j+1}  \in C\}$ has codimension $j+1$ in $V$. It follows that $\dim \{ [C] \in  V \; | \; x_1\ldots, x_g  \in C\}=0$ for general points $x_1, \ldots, x_g \in C_0$. Hence, for general points $x_1, \ldots, x_{g+1} \in C_0$, we have $\{ [C] \in  V \; | \; x_1\ldots, x_{g+1}  \in C\}=\{C_0\},$
and the generic injectivity of $f$ follows since $k'+\eps =g+r
 \geq g+1$.

The image of $f$ does not lie in the singular locus of $\Sym^{k'+\eps}(S)$. Hence, its inverse image under the Hilbert-Chow morphism is a ($k'+\eps+g$)-dimensional subvariety
of $S^{[k'+\eps]}$ that is birational to a $\PP^{r}$-bundle. 
As above,  the natural rational map  $S^{[k'+\eps]} \x S^{[k-k']} \dashrightarrow  S^{[k+\eps]}$ maps $\im f$ to a codimension $r$ subvariety of $S^{[k+\eps]}$. 
Since the Albanese map is constant on each rational subvariety of $S^{[k+\eps]}$, we obtain a subvariety
$W_{r,k'} \subset S^{[k]}_{\eps}$ of codimension 
$r$
that is birational to a $\PP^r$-bundle and the maximal rational quotient of the desingularization of $W_{r,k'}$ has dimension $2(k-r)$ by \cite[Thm. 4.4]{AV}.  
The coefficients in the class of the lines in the $\PP^r$-fibres of $W_{r,k'}$ with respect to the canonical decomposition \eqref{eq:N1} are the same as the ones of $R_{p,\delta,k'}=R_{p,p+r-k'-\eps,k'}=L-[2(k'+\eps)-r-1]\mathfrak{r}_{k'}$. This concludes the proof.
\end{proof}

We conclude with an interesting example, where Theorem \ref{thm:contrazioni} provides an immersion $\mathbb{P}^k\hookrightarrow S^{[k]}_\eps$.

\begin{example} \label{ex:romarosica}
By Remark \ref{rem:cile1} , when $\delta=0$ the class $R_{p,0,k}$ has the minimal possible self-intersection (i.e., $q(R_{p,0,k})=-(k+3-2\eps)/2$) if and only if $\alpha=1$ and $p= 2(k-1)+5\eps$. 
We assume these numerical conditions are satisfied and show, by explicit construction, that $R_{p,0,k}$ is the class of a line moving in a $\mathbb{P}^k \subset S_{\eps}^{[k]}$. Note that in this case $\dim\M=2\eps$ and condition \eqref{eq:nuovobound} is satisfied. With the same notation introduced just before Theorem \ref{thm:contrazioni}, $\P^\circ=\P$ and any component of $g^{-1}(S^{[k]}_\eps)$ is isomorphic to $\PP H^0(\E)$ for some vector bundle $[\E]\in\M$. We consider the restricted morphism
\[ 
\xymatrix{ 
\overline{g}:=g|_{\mathbb{P} H^0(\E)}:\mathbb{P}^k=\mathbb{P} H^0(\E)  \ar[r] & S_{\eps}^{[k]}\sub S^{[k+\eps]},
}
\]
which is injective by Theorem \ref{thm:contrazioni}  and, more precisely, an embedding by the result below.
\end{example}

\begin{prop}\label{theend}
Let $(S,L)$ be a general primitively polarized symplectic surface of genus $p=2(k-1)+5\eps$ for an integer $k\geq 2$. Then the map $\overline{g}$ is an embedding. In particular, the class $R_{p,0,k}$ is the class of a line in a $\mathbb{P}^k\subset S^{[k]}_\eps$.
\end{prop}

\begin{proof}
It is enough to show that for all $[s]\in\mathbb{P}H^0(\E)$ the differential
$$
d\overline{g}_{[s]}:T_{[s]}\mathbb{P}(H^0(\E))\to T_{[Z]} S^{[k+\eps]}
$$ 
is injective, where $Z$ is the zero scheme of $s$. First of all, we recall the isomorphisms $$T_{[s]}\mathbb{P} H^0(\E)\simeq\Hom(\mathbb{C}, H^0(\E )/\langle s\rangle)\simeq H^0(\E)/\langle s\rangle$$ and $T_{[Z]} S^{[k+\eps]}=H^0(\N_{Z/S})$, and use them in order to describe $d\overline{g}_{[s]}$. Given $t\in H^0( \E )/\langle s\rangle$, the evaluation map $ev:\langle t,s\rangle\otimes \O_S\to\E$ is injective and drops rank along a curve $\Gamma_t\in\{L\}$ containing $Z$. We denote by $B$ the cokernel of $ev$, which is supported on $\Gamma_t$. Since $h^2(\E)=h^1(\E)=0$ by Proposition \ref{prop:importante}, we obtain isomorphisms:
\begin{equation}\label{identificazione}
\langle t,s\rangle \simeq H^1(B)\simeq H^0(B_1)^\vee,
\end{equation}
where $B_1:=\ext^1_{ \O_S}(B,\O_S)$. By  \eqref{eq:andreas3}, one has $B_1=\hom_{\O_{\Gamma_t}}(I_{Z/\Gamma_t},\O_{\Gamma_t})$. Let  $\sigma_s,\,\sigma_t\in  H^0(B_1)$ denote the duals of the images of $s$ and $t$ under \eqref{identificazione}, and consider the following short exact sequence:
\[
\xymatrix{
0 \ar[r] &  \O_{\Gamma_t} \ar[r]^{\hspace{-1.3cm}\sigma_s} &  \hom_{\O_{\Gamma_t}}(I_{Z/\Gamma_t},\O_{X_t}) \ar[r]^{\alpha_t} &  \hom_{\O_{\Gamma_t}}(I_{Z/\Gamma_t},\O_{Z})
\ar[r] &  0.
}
\]
Then we have $0\neq H^0(\alpha_t)\sigma_t\in H^0(\hom_{O_{\Gamma_t}}(I_{Z/\Gamma_t},\O_{Z}))=H^0(\N_{Z/{\Gamma_t}}) \hookrightarrow  H^0(\N_{Z/S})$,  where the last inclusion is given by taking cohomology in 
\begin{equation} \label{eq:osvaldo}
\xymatrix{
0 \ar[r] &  \N_{Z/\Gamma_t} \ar[r]^{\hspace{-0.1cm}\iota_t} \ar[r] &  \N_{Z/S} \ar[r] & \N_{\Gamma_t/S}|_{Z} \ar[r] &    0.
}
\end{equation}
\color{black}
By construction, $d\overline{g}_{[s]}(t)=H^0(\iota_t)\circ H^0(\alpha_t)\sigma_t\in H^0(\N_{Z/S})$.

We are now able to prove the injectivity of $d\overline{g}_{[s]}$. Let $t, t' \in 
 H^0(\E)/\langle s\rangle$ such that 
$d\overline{g}_{[s]}(t)=d\overline{g}_{[s]}(t')$.
We first assume that $\Gamma_t\simeq \Gamma_{t'}$. Since $h^1(\E)=0$, then the natural map $h:G(2,H^0(\E))\to \G^1_{k+1}(\vert L\vert)$ is injective; indeed, any $V\in G(2,H^0(\E))$ defines a short exact sequence like \eqref{eq:conv} and 
and it is enough to tensor it with $\E^\vee$ in order to conclude that $\mathbb{P}(\Hom(\E,B))$ is a point. Since $h(\langle s, t\rangle)=h(\langle s, t'\rangle)=(\Gamma_t, \O_{\Gamma_t}(Z))$, we have $\langle s, t\rangle=\langle s, t'\rangle $, that is, $t'=\lambda t$ for some $\lambda\in \mathbb{C}$. Then, $\sigma_{\lambda t}=\lambda\sigma_t\neq \sigma_t$ unless $\lambda=1$. Therefore, we can assume $\Gamma_{t'}\not\simeq \Gamma_t$; it is then clear from \eqref{eq:osvaldo} that the section $H^0(\iota_{t'})\circ H^0(\alpha_{t'})\sigma_{t'}\in H^0(\N_{Z/S})$ does not lie in the image of $H^0(\iota_{t})$. This concludes the proof.
\end{proof}

\end{document}